\documentclass[12pt]{amsart}
\usepackage{amsmath, amsthm, amssymb,amscd, mathrsfs, mathtools,tikz-cd, subcaption}
\usepackage{graphicx}
\usepackage{listings}
\usepackage{makecell}
\usepackage{float}

\usepackage[all]{xy}

\usepackage[bookmarks=true,
            bookmarksnumbered=true, breaklinks=true,
            pdfstartview=FitH, hyperfigures=false,
            plainpages=false, naturalnames=true,
            colorlinks=true,pagebackref=true,
            pdfpagelabels]{hyperref}

\newtheorem{thm}{Theorem}[section]
\newtheorem{pro}[thm]{Proposition}
\newtheorem{lem}[thm]{Lemma}

\newtheorem{conj}[thm]{Conjecture}

\theoremstyle{definition}

\newtheorem{rmk}[thm]{Remark}

\newtheorem{ex}[thm]{Example}

\allowdisplaybreaks

\newcommand\Hom{\operatorname{Hom}}

\def\Inn{\operatorname{Inn}}

\def\Conj{\operatorname{Conj}}
\def\Core{\operatorname{Core}}

\title{Enhancements of link colorings via idempotents of quandle rings}

\author{Mohamed Elhamdadi } 
\address{Department of Mathematics, 
University of South Florida, Tampa, FL 33620, U.S.A.} 
\email{emohamed@math.usf.edu} 

\author{Brandon Nunez}
\address{Department of Mathematics, 
University of South Florida, Tampa, FL 33620, U.S.A.} 
\email{nunez7@usf.edu} 

\author{Mahender Singh}
\address{Department of Mathematical Sciences,
Indian Institute of Science Education and Research (IISER) Mohali, Sector 81, SAS Nagar, P O Manauli, Punjab 140306, India} 
\email{mahender@iisermohali.ac.in}


\begin{document}
\maketitle

\begin{abstract}
We show that quandle rings and their idempotents lead to proper enhancements of the well-known quandle coloring invariant of links in the 3-space. We give explicit examples to show that the new invariants are also stronger than the $\Hom$ quandle invariant when the coloring quandles are medial. We provide computer assisted computations of idempotents for all quandles of order less than six, and also determine  quandles for which the set of all idempotents is itself a quandle. The data supports our conjecture about triviality of idempotents of integral quandle rings of finite latin quandles. We also determine Peirce spectra for complex quandle algebras of some small order quandles.

\end{abstract}

\section{Introduction}
Quandle rings are analogues of group rings and were introduced in \cite{MR3977818} as an attempt to bring ring theoretic techniques to quandle theory. Quandle rings of non-trivial quandles are non-associative, and it has been proved \cite{MR3915329} that these rings are not even power-associative, which is the other extreme of the spectrum of associativity. Moreover, it has been shown that quandle rings of non-trivial quandles over rings of characteristic other than two and three cannot be alternative or Jordan algebras  \cite{BPS1}. The isomorphism problem for quandle rings has been explored, and examples of non-isomorphic finite quandles with isomorphic quandle rings have been given in \cite{MR3915329}. Moreover, it has been proved that if two finite quandles admit doubly transitive actions by their inner automorphisms and have isomorphic quandle rings, then the quandles have the same number of orbits of each size. A cohomological study of quandle rings has also been initiated in a recent work \cite{EMSZ2022}, where a complete characterization of derivations for quandle algebras of dihedral quandles over fields of characteristic zero has been given. Zero-divisors in quandle rings have been investigated in \cite{BPS1} using the idea of orderability of quandles. It has been proved that quandle rings of left or right orderable quandles which are semi-latin have no zero-divisors. As a special case, it follows that quandle rings of free quandles have no zero-divisors. It is well-known that units play a fundamental role in the structure theory of group rings. In contrast, idempotents are the most natural objects in quandle rings since each quandle element is, by definition, an idempotent of the quandle ring. Exploiting the idea of a quandle covering as developed in \cite{MR1954330, MR3205568}, the complete set of idempotents of the quandle ring of an involutory covering, when the quandle ring of the base has only trivial idempotents, has been determined in \cite{ENSS2022}. Interestingly, the set of all idempotents in this case is itself a quandle with respect to the ring multiplication. We shall see that this phenomenon occurs more often in mod 2 quandle rings. In addition, integral quandle rings of free quandles have been shown to admit only trivial idempotents \cite{ENSS2022}, which gives an infinite family of quandles with this property.
\par

The purpose of this computational paper is two fold. The first being new applications of quandle rings to knot theory, and the second being a deeper exploration of idempotents in quandle rings. Section \ref{prelim} gives a quick review of quandle rings. In Section \ref{idempotents of quandle rings}, we  make some basic observations about idempotents of quandle rings which will be used in later sections. In sections \ref{sec 4} and \ref{sec 5},  we prove that quandle rings and their sets of idempotents lead to two proper enhancements of the well-known quandle coloring invariant of links. Explicit examples are given to show that the new invariants are even stronger than the $\Hom$ quandle invariant when the coloring quandles are medial (Theorem \ref{ring enhancement} and Theorem \ref{idempotent enhancement}). In Section \ref{computer computations}, we present computer assisted calculations of all idempotents for integral as well as mod 2 quandle rings of all quandles of order less than six, and also determine quandles for which the set of all idempotents is itself a quandle. The data also supports our conjecture \cite[Conjecture 3.10]{ENSS2022} about triviality of idempotents of integral quandle rings of finite latin quandles, and also suggests that any non-zero idempotent of an integral quandle ring has augmentation value 1 (Conjecture \ref{augmentation one conjecture}).  Finally, in Section \ref{sec Peirce spectra}, we discuss Peirce spectra for complex quandle algebras of quandles of order three. We prove that the Peirce spectrum of a complex quandle algebra can be as large as possible when the underlying quandle is non-latin, and that the Peirce spectrum for the dihedral quandle of order three is $\{0, 1, -1 \}$ (Theorem \ref{Peirce spectra 3 order quandle}). Our computer assisted computations suggest that the Peirce spectra is the same for each dihedral quandle of odd order (Conjecture \ref{Peirce spectra conjecture}).

\medskip

\section{Review of quandle rings}\label{prelim}
A {\it quandle} is a non-empty set $X$ with a binary operation $*$ such that each map $S_y: X \rightarrow X$, given by $S_y(x) =x*y$ for $x \in X$, is an automorphism of $X$ fixing $y$. Note that right multiplications being homomorphisms of $X$ is equivalent to the right distributivity axiom $$(x*y)*z=(x*z)*(y*z)$$ for all $x, y, z \in X$. It is well-known that quandle axioms are simply algebraic formulations of the three Reidemeister moves of planar diagrams of knots and links in the 3-space \cite{MR2628474, MR0672410}, which therefore makes them suitable for defining invariants of knots and links.  A quandle $X$ is {\it trivial} if $x*y=x$ for all $x, y\in X$. The group $\Inn(X)= \langle S_x \mid x \in X \rangle$ generated by all right multiplications is called the {\it inner automorphism group} of $X$.
\par

Conjugacy classes in groups are a rich source of quandles. Each group $G$ can be turned into a quandle $\Conj(G)$ with the binary operation $x*y= yx y^{-1}$, and called the {\it conjugation quandle}. Similarly, every group $G$ can be turned into a quandle $\Core(G)$ by setting $x*y=yx^{-1}y$, and referred as the {\it core quandle}. In particular, the cyclic group of order $n> 2$ gives the dihedral quandle $R_n$ of order $n$. 
\par
A quandle $X$ is called {\it latin} if each left multiplication $L_x: X \to X$, given by $L_x(y)=x*y$ for $y \in X$, is a bijection of $X$. For example, dihedral quandles of odd orders and commutative quandles are latin. A quandle $X$ is called {\it medial} if $$(x*y)*(z*w)=(x*z)*(y*w)$$ for all $x, y, z, w \in X$. The identity stems from the fact that these are precisely the quandles for which the map $X \times X \to X$ given by $(x, y) \mapsto x*y$ is a quandle homomorphism, where $X \times X$ is equipped with the usual product structure. A quandle is {\it involutory} if each right multiplication is of order two. For example, the core quandle of any group is medial as well as involutory.
\par

Let $(X, *)$ be a quandle and $\mathbf{k}$ an integral domain with unity {\bf 1}.  Let $e_x$ be a unique symbol corresponding to each $x \in X$. Let $\mathbf{k}[X]$ be the set of all formal expressions of the form $\sum_{x \in X }  \alpha_x e_x$, where $\alpha_x \in \mathbf{k}$ such that all but finitely many $\alpha_x=0$.  The set $\mathbf{k}[X]$ has a free $\mathbf{k}$-module structure with basis $\{e_x \mid x \in X \}$ and admits a product given by 
 $$ \Big( \sum_{x \in X }  \alpha_x e_x \Big) \Big( \sum_{ y \in X }  \beta_y e_y \Big)
 =   \sum_{x, y \in X } \alpha_x \beta_y e_{x * y},$$
where $x, y \in X$ and $\alpha_x, \beta_y \in \mathbf{k}$. This turns $\mathbf{k}[X]$ into a ring (rather a $\mathbf{k}$-algebra when  $\mathbf{k}$ is a field) called the {\it quandle ring} of $X$ with coefficients in $\mathbf{k}$. Even if the coefficient ring $\mathbf{k}$ is associative, the quandle ring $\mathbf{k}[X]$ is non-associative when $X$ is a non-trivial quandle. The quandle $X$ can be identified as a subset of $\mathbf{k}[X]$ via the natural map $x \mapsto {\bf 1}e_x=e_x$. Throughout the text, we assume that the coefficient ring $\mathbf{k}$ is an integral domain with unity, unless mentioned otherwise.
\medskip

\section{Idempotents of quandle rings}\label{idempotents of quandle rings}
Let $X$ be a quandle and $\mathbf{k}$ an integral domain with unity.  A non-zero element $u \in \mathbf{k}[X]$ is called an {\it idempotent} if $u^2=u$. We denote the set of all idempotents of $\mathbf{k}[X]$ by $\mathcal{I}(\mathbf{k}[X])$. Clearly,  elements of the basis $\{e_x \mid x \in X\}$ are idempotents of $\mathbf{k}[X]$ and referred as {\it trivial idempotents}.  A non-trivial idempotent  is an element of $\mathbf{k}[X]$ that is not of the form $e_x$ for any $x \in X$. Since the construction of a quandle ring is functorial, a quandle homomorphism $X \rightarrow Y$ induces a ring homomorphism $\mathbf{k}[X] \rightarrow \mathbf{k}[Y]$, which therefore maps $ \mathcal{I} (\mathbf{k}[X] )$ into $\mathcal{I} (\mathbf{k}[Y] )$.  Let  $\varepsilon: \mathbf{k}[X] \rightarrow \mathbf{k}$ given by $$\varepsilon \Big(\sum_{x \in X }  \alpha_x e_x \Big)=\sum_{x \in X }  \alpha_x$$
be the {\it augmentation map}. Since $\varepsilon$ is a ring homomorphism, it follows that each idempotent of $\mathbf{k}[X]$ has augmentation value 0 or 1.
\par

We begin with the following basic observation.

\begin{pro}\label{prop commutative quandle}
If $X$ is a commutative quandle, then
$$\mathcal{I}(\mathbb{Z}_2[X])= \Big\{ \sum_{x \in F} e_x  \mid F~\textrm{is a finite subset of}~X\Big\}.$$
In particular, if $X$ is finite, then $|\mathcal{I}(\mathbb{Z}_2[X])|=2^{|X|}-1$.
\end{pro}

\begin{proof}
Since $X$ is commutative, it follows that $\mathbb{Z}_2[X]$ is a commutative ring. If $F$ is a finite subset of $X$, then a direct check shows that the element $u= \sum_{x \in F} e_x$ satisfy $u=u^2$. The second assertion follows immediately. 
\end{proof}

The notion of a quandle covering \cite{MR1954330, MR3205568} can be used to identify idempotents for a large family of quandle rings. Recall that a surjective quandle homomorphism $p: X  \to Y$ is called a {\it quandle covering} if $S_x=S_{x'}$ whenever $p(x) = p(x')$ for any $x, x' \in X$.  For example, a surjective group homomorphism $p: G \to H$ yields a quandle covering $\Conj(G) \to \Conj(H)$ if and only if $\ker(p)$ is a central subgroup of $G$. Similarly, a surjective group homomorphism $p: G\to H$ yields a quandle covering $\Core(G) \to \Core(H)$ if and only if $\ker(p)$ is a central subgroup of  $G$ of exponent two.
\par

The following result  \cite[Corollary 4.8]{ENSS2022} has been proved using the fact that the modulo $n$ reduction $R_{2n} \to R_n$ is a quandle covering.

\begin{pro}\label{covering idempotents}
Let $n \ge 3$ be an odd integer. If $\mathbf{k}[R_{n}]$ has only trivial idempotents,  then
\begin{small}
$$\mathcal{I}(\mathbf{k}[R_{2n}]) =\Big\{ \big(\beta e_j + (1-\beta)e_{n+j}\big)+ \sum_{i=0}^{n-1} \alpha_i \big(e_i-e_{n+i} + e_{2j-i}- e_{n+2j-i}\big) ~\bigl\vert ~ 0 \le j \le n-1 ~\textrm{and~} ~~ \alpha_i, \beta \in \mathbf{k} \Big\}.$$
\end{small}
 \end{pro}
 
 We shall need the following lemmas in sections \ref{sec 4} and \ref{sec 5}.

\begin{lem}\label{r3 and r5 idempotents}
The quandle ring $\mathbb{Z}[R_3]$ has only trivial idempotents.
\end{lem}

\begin{proof}
Note that idempotents of $\mathbb{Z}[R_3]$ are computed in \cite{BPS1}, and we present a direct proof here for the sake of completeness. Let $u=\alpha e_0 + \beta e_1 + \gamma e_2 \in \mathbb{Z}[R_3]$ be a non-zero idempotent. Comparing coefficients of the basis elements in the equality $u=u^2$ gives $$\alpha = \alpha^2 + 2 \beta \gamma, \quad \beta = \beta^2 + 2 \gamma \alpha \quad \textrm{and} \quad \gamma = \gamma^2 + 2 \alpha  \beta.$$ Clearly, all the coefficients cannot be equal. Suppose that precisely two coefficients are equal. If $\alpha=\beta$, then we obtain $\gamma = \gamma^2 + 2 \alpha^2$. This implies that $\alpha=0$ and  $\gamma=1$, which gives $u=e_2$. Similarly, $\beta=\gamma$ gives $u=e_0$, and $\alpha=\gamma$ gives $u=e_1$. Now, suppose that all the coefficients are distinct. Subtracting the second equation from the first gives $\alpha -\beta = \alpha^2 -\beta^2 + 2 \gamma(\beta -\alpha)$. Since $\alpha  \ne \beta$, we obtain $1 = (\alpha +\beta + \gamma) - 3 \gamma$. Since $\varepsilon(u)= 0$ or 1, we obtain $\gamma=0$. Similarly, by subtracting the other pairs of equations, we deduce that $\alpha=0$ and $\beta=0$, a contradiction. Hence $\mathbb{Z}[R_3]$ has  only trivial idempotents.
\end{proof}

\begin{lem}\label{order five quandle}
Let $X$ be the medial quandle with multiplication table 
$$X=
\begin{pmatrix} 
1 & 1 & 1 & 2 & 2 \\
2 & 2 & 2 & 3 & 3 \\
3 & 3 & 3 & 1 & 1 \\
5 & 5 & 5 & 4 & 4 \\
4 & 4 & 4 & 5 & 5 \\
\end{pmatrix}.
$$
Then the set of idempotents of $\mathbb{Z}_2[X]$ is 
\begin{small}
$$\mathcal{I}(\mathbb{Z}_2[X])= \big\{e_1,e_2,e_3,e_4,e_5,e_1+e_2+e_3,e_1+e_4+e_5,e_2+e_4+e_5, e_3+e_4+e_5,e_1+e_2+e_3+e_4+e_5 \big\}.$$ Further, $\mathcal{I}(\mathbb{Z}_2[X])$ is a medial quandle with respect to the ring multiplication.
\end{small}
\end{lem}

\begin{proof}
A direct check shows that an element of the form $e_i+e_j$ is not an idempotent.  Let us consider idempotents built up with three basis elements.  We can check that $e_1+e_2+e_4$, $e_1+e_2+e_5$, $e_2+e_3+e_4$, $e_2+e_3+e_5$, $e_1+e_3+e_4$ and $e_1+e_3+e_5$ are not idempotents, whereas $e_1 +e_2+e_3$, $e_1+e_4+e_5$, $e_2+e_4+e_5$ and $e_3+e_4+e_5$ are all idempotents. We can also see that no element built up with four basis elements is an idempotent, and that $e_1+e_2+e_3+e_4+e_5$ is clearly an idempotent. Hence, the set $\mathcal{I}(\mathbb{Z}_2[X])$ is as desired. That $\mathcal{I}(\mathbb{Z}_2[X])$ is a quandle follows from the fact that, for this quandle, the right multiplication by each idempotent is simply the automorphism of $\mathbb{Z}_2[X]$ induced by an inner automorphism of $X$. Lastly, mediality of $\mathcal{I}(\mathbb{Z}_2[X])$ follows directly from mediality of $X$.
\end{proof}

It is interesting to note that the set of all idempotents has the structure of a quandle in many cases.  If $X$ is a trivial quandle, then \cite[Proposition 4.1]{BPS1} gives

\begin{equation}\label{idempotents trivial quandle}
\mathcal{I}(\mathbf{k}[X] )= \Big\{ \sum_{x \in F} \alpha_x e_x \mid  F~\textrm{is a finite subset of}~X~~\textrm{and}~\sum_{x \in F} \alpha_x =1 \Big\}.
\end{equation}
A direct check shows that $\mathcal{I}(\mathbf{k}[X] )$ is a trivial quandle with respect to the ring multiplication (see also \cite[Corollary 4.12]{ENSS2022}). More generally, the following is a consequence of \cite[Theorem 4.5]{ENSS2022}.

\begin{pro}
Let $p : X \to Y$ be a quandle covering such that $X$ is involutory and $\mathbf{k}[Y]$ has only trivial idempotents. Then the set of all idempotents of $\mathbf{k}[X]$ is a quandle with respect to the ring multiplication.
\end{pro}

The following useful observations are immediate.

\begin{pro}
The following holds in $\mathbf{k}[X]$:
\begin{enumerate}
\item If $X$ is a medial quandle, then right and left multiplications by idempotents are distributive.
\item If $X$ is a medial quandle, then $\mathcal{I}(\mathbf{k}[X])$ is closed with respect to the ring multiplication.
\item If $\mathcal{I}(\mathbf{k}[X])$ is a quandle, then each idempotent has augmentation value 1.
\end{enumerate}
\end{pro}

\begin{proof}
Since $X$ is medial, it follows that $(uv)(wz)=(uw)(vz)$ for all $u, v, w, z \in \mathbf{k}[X]$. If $w$ is an idempotent, then $(uv)w=(uv)(ww)= (uw)(vw)$ and $w(uv)=(ww)(uv)=(wu)(wv)$, which proves (1). If $u$ and $v$ are idempotents, then $(uv)^2=(uv)(uv)=(uu)(vv)=uv$, which is assertion (2). Since the right multiplication by an augmentation 0 idempotent cannot be surjective, assertion (3) follows.
\end{proof}

 \medskip

\section{An enhancement of the coloring invariant by quandle rings}\label{sec 4}

Let $K$ be a link in $\mathbb{R}^3$, $Q(K)$ its link quandle and $X$ a quandle. Then the number of quandle homomorphisms $|\Hom(Q(K), X)|$ is an invariant of $K$ called the {\it quandle coloring invariant}. This invariant generalises the well-known Fox coloring invariant of links.  A link invariant which determines the quandle coloring invariant is called an {\it enhancement} of the quandle coloring invariant. Furthermore, an enhancement is called {\it proper} if there are examples in which the enhancement distinguishes links which have the same quandle coloring invariant. 
\par
If the target quandle is medial, then the set of all homomorphisms itself has a quandle structure \cite[Theorem 4.1]{MR3197054}.

\begin{thm}\label{sam result}
Let $X$ and $Y$ be quandles such that $Y$ is medial. Then $\Hom(X, Y)$ is a medial quandle under the point-wise product.
\end{thm}

Considering homomorphisms from link quandles to medial quandles, the preceding result is an enhancement of the quandle coloring invariant.  If $X$ is not connected, then the quandle structure on  $\Hom(X, Y)$ is not determined by $|Y|$. Consequently, for links $K$ with two or more components, the quandle $\Hom(Q(K), Y)$ is a stronger invariant than the coloring invariant $|\Hom(Q(K), Y)|$. See \cite{MR3197054} for related examples.
\par

Let $X$ and $Y$ be quandles and $\Hom_{alg}\big(\mathbf{k}[X], \mathbf{k}[Y]\big)$ denotes the set of $\mathbf{k}$-algebra homomorphisms from $\mathbf{k}[X]$ to  $\mathbf{k}[Y]$.

\begin{thm}\label{ring enhancement}
If $K$ is a link and $X$ a quandle, then $|\Hom_{alg}\big(\mathbf{k}[Q(K)], \mathbf{k}[X]\big)|$ is a proper enhancement of the quandle coloring invariant $|\Hom(Q(K), X)|$. Further, if $X$ is medial, then $|\Hom_{alg}\big(\mathbf{k}[Q(K)], \mathbf{k}[X]\big)|$ is a proper enhancement of $\Hom(Q(K), X)$.
\end{thm}

\begin{proof}
Since every homomorphism of quandles induces a $\mathbf{k}$-linear homomorphism of their quandle rings, there is an injection of $\Hom(Q(K), X)$ into $\Hom_{alg}\big(\mathbf{k}[Q(K)], \mathbf{k}[X]\big)$, and hence $|\Hom_{alg}\big(\mathbf{k}[Q(K)], \mathbf{k}[X]\big)|$ is an enhancement of  $|\Hom(Q(K), X)|$. The following example shows that this invariant is, in fact, a proper enhancement of the quandle $\Hom(Q(K), X)$ when $X$ is medial. 

\begin{ex}
Consider braid diagrams of links $K_1= L4a1\{0\}$ and $K_2=L5a1\{1\}$ as in Figure \ref{fig2} (see \cite{knot info} for the notational convention).
 \begin{figure}[!ht]
 \begin{center}
\includegraphics[height=6cm, width=8cm]{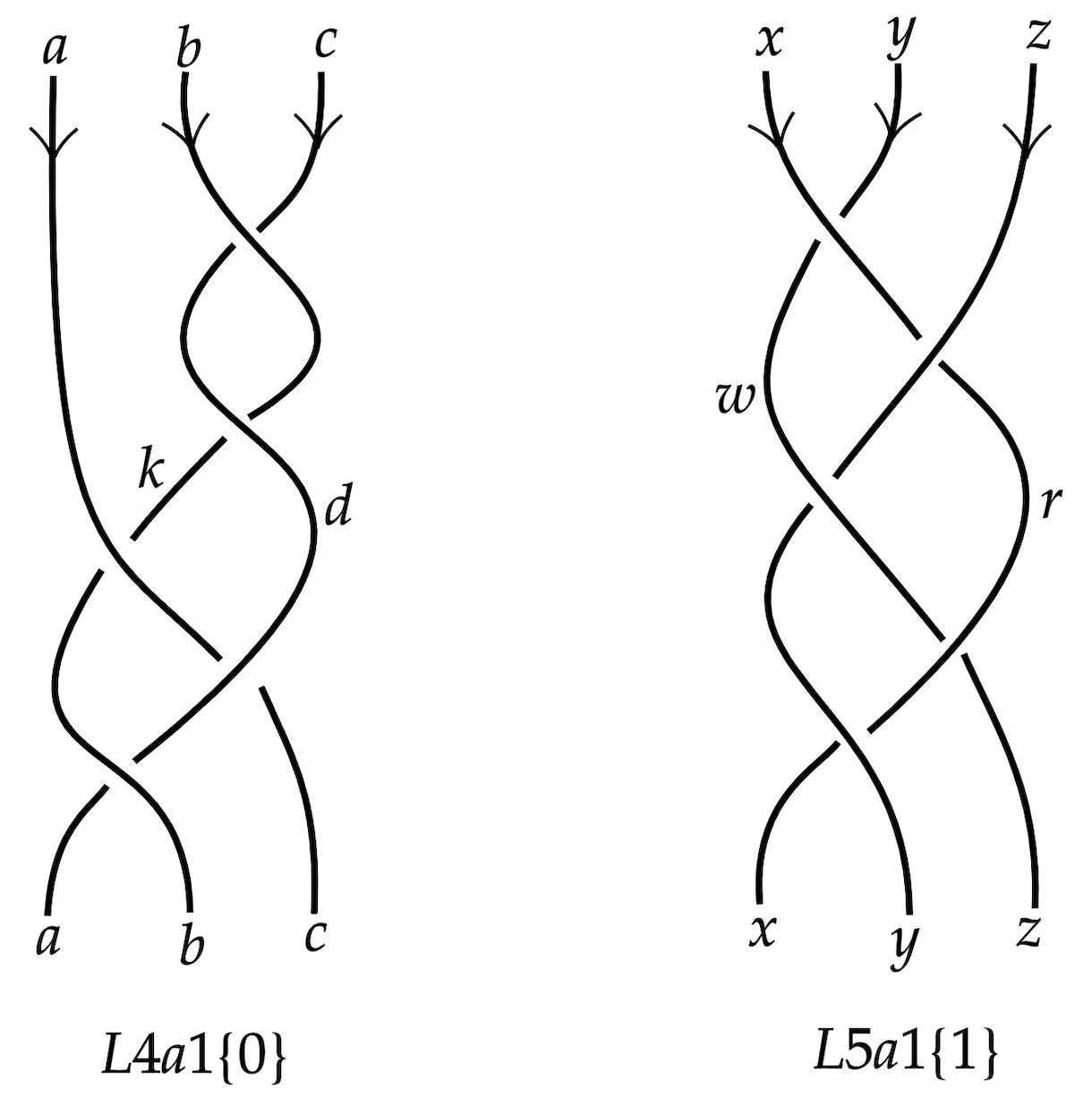}
\end{center}
\caption{Braid diagrams of links  $K_1=L4a1\{0\}$ and $K_2=L5a1\{1\}$.} \label{fig2}
\end{figure}
Using the labellings as in Figure \ref{fig2}, we obtain presentations of knot quandles $Q(K_1)$ and  $Q(K_2)$ as 
\begin{eqnarray}
\label{eq prest 5} \quad \quad Q(K_1) &=& \big\langle a, b, c, d, k ~\mid~ d*b=c,~ k*d=b,~ b*a=k,~ a*d=c, ~a*b=d \big\rangle,\\
\label{eq prest 6} \quad  \quad Q(K_2) &=& \big\langle x, y, z, r, w ~\mid~ w*x=y, ~x*z=r, ~y*w=z,~ w*r=z,~ x*y=r \big\rangle.
\end{eqnarray}
An analogue of Tietze's theorem relating two finite presentations of a quandle is known due to Fenn and Rourke \cite[Theorem 4.2]{MR1194995}.  Applying Tietze transformations to \eqref{eq prest 5} and \eqref{eq prest 6} gives
\begin{eqnarray}
\label{eq prest 7} \quad \quad && Q(K_1) = \big\langle a, b ~\mid~ a*(a*b)=(a*b)*b, ~ (b*a)*(a*b)=b \big\rangle,\\
\label{eq prest 8} \quad \quad  && Q(K_2) = \big\langle x, w ~\mid~ w*(x*(w*x))=(w*x)*w,~x*((w*x)*w)=x*(w*x) \big\rangle.
\end{eqnarray}
We color $K_1$ and $K_2$ by the dihedral quandle $R_6= \{0, 1, 2, 3, 4, 5 \}$. If $f:Q(K_1) \to R_6$, then relations in \eqref{eq prest 7} show that $2f(a)=2f(b)$, and hence we have $|\Hom(Q(K_1), R_6)|=12$. Similarly, if $f:Q(K_2) \to R_6$, then relations in \eqref{eq prest 7} imply that $2f(x)=2f(v)$, and hence $|\Hom(Q(K_2), R_6)|=12$. In fact,  we have 
$$\Hom(Q(K_1), R_6) \cong \Hom(Q(K_2), R_6)$$ as medial quandles. 
\par
We claim that $|\Hom_{alg}\big(\mathbb{Z}[Q(K_1)], \mathbb{Z}[R_6]\big)| \neq |\Hom_{alg}\big(\mathbb{Z}[Q(K_2)], \mathbb{Z}[R_6]\big)|$. It follows from Lemma \ref{r3 and r5 idempotents} that $\mathbb{Z}[R_3]$ has only trivial idempotents. Hence, by Proposition \ref{covering idempotents}, we have
\begin{small}
$$ \mathcal{I}\big(\mathbb{Z}[R_6]\big) =\Big\{ \big(\beta e_j + (1-\beta)e_{j+3}\big)+ \sum_{i=0}^{2} \alpha_i \big(e_i-e_{i+3} + e_{2j-i}- e_{2j-i+3}\big) ~\bigl\vert 
~ 0 \le j \le 2 ~~\textrm{and} ~~ \alpha_i, \beta \in \mathbb{Z} \Big\}.$$
\end{small}
If $f \in \Hom_{alg}\big(\mathbb{Z}[Q(K_1)], \mathbb{Z}[R_6]\big)$, then relations in \eqref{eq prest 7} give
\begin{eqnarray}
\label{eq of f9}  f(e_a)\big(f(e_a) f(e_b)\big) &=&\big(f(e_a)f(e_b)\big) f(e_b),\\
\label{eq of f10}  \big(f(e_b)f(e_a)\big) \big(f(e_a) f(e_b)\big) &=& f(e_b).
 \end{eqnarray}
If $f(e_a)=0$, then \eqref{eq of f10} implies that $f(e_b)=0$. On the other hand, if $f(e_b)=0$, then  $f(e_a)$ can take any value from $\mathcal{I}(\mathbb{Z}[R_6]) \cup \{ 0\}$. 
\par
We now consider the case when $f(e_a), f(e_b)$ lie in $\mathcal{I}(\mathbb{Z}[R_6])$. Let us write $u_j= \big(\beta e_j + (1-\beta)e_{j+3}\big) + v_j$ for $j=0, 1, 2$, where $v_j= \sum_{i=0}^{2} \alpha_i \big(e_i-e_{i+3} + e_{2j-i}- e_{2j-i+3}\big)$. If $f(e_a)= u_j$ and  $f(e_b)=u_k$ for $j \ne k$, then we see that 
$$f(e_a)\big(f(e_a) f(e_b)\big) = u_j (u_j u_k)=u_j u_{2k-j}=u_{k},$$
whereas
$$\big(f(e_a)f(e_b)\big) f(e_b)=(u_j u_k) u_k= u_{2k-j}u_k= u_{j}.$$
Hence, $f$ does not satisfy the relation \eqref{eq of f9}. Suppose that $f(e_a)= u_j$ and  $f(e_b)=u_j'$ for some $j=0, 1, 2$, where $u_j= \big(\beta e_j + (1-\beta)e_{j+3}\big) + v_j$,  $u_j'= \big(\beta' e_j + (1-\beta')e_{j+3}\big) + v_j'$, $v_j= \sum_{i=0}^{2} \alpha_i \big(e_i-e_{i+3} + e_{2j-i}- e_{2j-i+3}\big)$  and $v_j'= \sum_{i=0}^{2} \alpha_i' \big(e_i-e_{i+3} + e_{2j-i}- e_{2j-i+3}\big)$. Then, we see that
$$f(e_a)\big(f(e_a) f(e_b)\big) = u_j (u_j u_j')=u_j=\big(f(e_a)f(e_b)\big) f(e_b)$$
and
$$\big(f(e_b)f(e_a)\big) \big(f(e_a) f(e_b)\big) = (u_j' u_j)(u_j u_j')= u_j= f(e_b).$$
Hence, $f$ satisfies both relations \eqref{eq of f9}  and \eqref{eq of f10}.
\par

Next, we take $g \in \Hom_{alg}\big(\mathbb{Z}[Q(K_2)], \mathbb{Z}[R_6]\big)$. Then relations in \eqref{eq prest 8} give
\begin{eqnarray}
\label{eq of f11}  g(e_w) \big(g(e_x) \big(g(e_w)g(e_x) \big) \big) &=& \big(g(e_w)g(e_x) \big) g(e_w),\\
\label{eq of f12}  g(e_x) \big(\big(g(e_w) g(e_x) \big) g(e_w) \big) &=& g(e_x) \big(g(e_w) g(e_x) \big).
 \end{eqnarray}
If $g(e_x)=0$, then  $g(e_w)$ can take any value from $\mathcal{I}(\mathbb{Z}[R_6]) \cup \{ 0\}$. Similarly, if $g(e_w)=0$, then  $g(e_x)$ can take any value from $\mathcal{I}(\mathbb{Z}[R_6]) \cup \{ 0\}$. 
\par
Now consider the case when $g(e_x), g(e_w)$ lie in $\mathcal{I}(\mathbb{Z}[R_6])$. If $g(e_x)= u_j$ and  $g(e_w)=u_k$ for $j \ne k$, then we see that 
$$g(e_w) \big(g(e_x) \big(g(e_w)g(e_x) \big) \big) = u_k (u_j (u_k u_j)) =u_k (u_j u_{2j-k})= u_k u_{3j-2k}=u_{k}$$
and 
$$\big(g(e_w)g(e_x) \big) g(e_w)=(u_k u_j)u_k=u_{2j-k} u_k=u_{j}.$$
Hence, $g$ does not satisfy the relation \eqref{eq of f11}.  Suppose that $g(e_x)= u_j$ and  $g(e_w)=u_j'$ for some $j=0, 1, 2$, where $u_j= \big(\beta e_j + (1-\beta)e_{j+3}\big) + v_j$,  $u_j'= \big(\beta' e_j + (1-\beta')e_{j+3}\big) + v_j'$, $v_j= \sum_{i=0}^{2} \alpha_i \big(e_i-e_{i+3} + e_{2j-i}- e_{2j-i+3}\big)$  and $v_j'= \sum_{i=0}^{2} \alpha_i' \big(e_i-e_{i+3} + e_{2j-i}- e_{2j-i+3}\big)$. Then, we see that
$$g(e_w) \big(g(e_x) \big(g(e_w)g(e_x) \big) \big) =u_j'(u_j(u_j'u_j)) =u_j'= (u_j'u_j)u_j'=  \big(g(e_w)g(e_x) \big) g(e_w)$$
and
$$g(e_x) \big(\big(g(e_w) g(e_x) \big) g(e_w) \big) =u_j((u_j'u_j)u_j')=u_j= u_j(u_j'u_j)= g(e_x) \big(g(e_w) g(e_x) \big).$$
Hence, $g$ satisfies both relations \eqref{eq of f11} and \eqref{eq of f12}. This gives $|\Hom_{alg}\big(\mathbb{Z}[Q(K_1)], \mathbb{Z}[R_6]\big)| \neq |\Hom_{alg}\big(\mathbb{Z}[Q(K_2)], \mathbb{Z}[R_6]\big)|$, which proves our claim.
\end{ex}
This completes the proof of the theorem.
\end{proof}
\medskip

\section{An enhancement of the coloring invariant by idempotents}\label{sec 5}
In this section, we consider colorings of links by idempotents of quandle rings, and prove that it gives a proper enhancement of the quandle coloring as well as the $\Hom$ quandle invariant.

\begin{thm}\label{idempotent enhancement}
If $K$ is a link and $X$ a quandle such that $\mathcal{I}(\mathbf{k}[X])$ is a quandle with respect to the ring multiplication, then $|\Hom\big(Q(K),\mathcal{I}(\mathbf{k}[X]) \big)|$ is a proper enhancement of the quandle coloring invariant $|\Hom\big(Q(K), X\big)|$. Furthermore, if $X$ is medial, then the quandle $\Hom\big(Q(K),\mathcal{I}(\mathbf{k}[X]) \big)$ is a proper enhancement of $\Hom\big(Q(K), X\big)$.
\end{thm}

\begin{proof}
Since $X$ can be viewed as the subquandle of  trivial idempotents in $\mathcal{I}(\mathbf{k}[X])$, it follows that  $|\Hom\big(Q(K),\mathcal{I}(\mathbf{k}[X]) \big)|$ is an enhancement of  $|\Hom(Q(K), X)|$.  Further, if $X$ is medial, then $\mathcal{I}(\mathbf{k}[X])$ is also medial. It follows from Theorem \ref{sam result} that $\Hom\big(Q(K),\mathcal{I}(\mathbf{k}[X]) \big)$ is a medial quandle containing $\Hom(Q(K), X)$ as a medial subquandle, and hence is an enhancement. The following example shows that these invariants are proper enhancements.

\begin{ex}
Consider  braid diagrams of links $K_1=L2a1\{0\}$ and $K_2=L4a1\{1\}$ as in Figure \ref{fig3}.
 \begin{figure}[!ht]
 \begin{center}
\includegraphics[height=7.3cm, width=8cm]{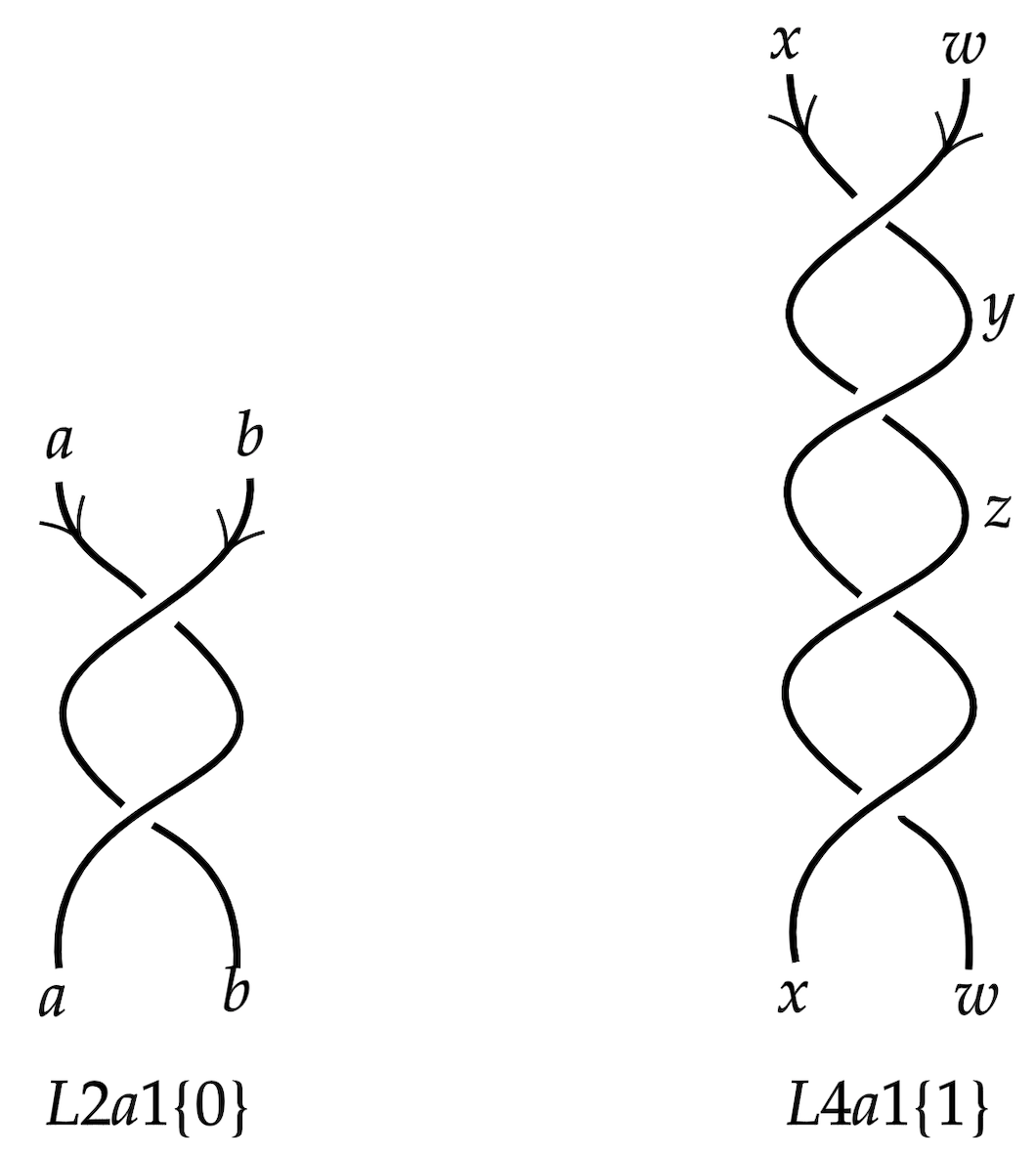}
\end{center}
\caption{Braid diagrams of links $K_1=L2a1\{0\}$ and $K_2=L4a1\{1\}$.} \label{fig3}
\end{figure}
Using the labellings as in Figure \ref{fig3}, we obtain presentations of knot quandles $Q(K_1)$ and  $Q(K_2)$ as 
\begin{eqnarray}
\label{eq prest 9} \quad \quad Q(K_1) &=& \big\langle a, b ~\mid~ a*b=a,~ b*a=b\big\rangle,\\
\label{eq prest 10} \quad  \quad Q(K_2) &=& \big\langle x, y, z, w ~\mid~ x*w=y,~ w*y=z,~ y*z=x,~ z*x=w \big\rangle.
\end{eqnarray}
We reduce the presentation \eqref{eq prest 10} to
\begin{eqnarray}
\label{eq prest 11} \quad  \quad Q(K_2) &=& \big\langle x, w ~\mid~ (x*w)* (w*(x*w))=x,~ (w*(x*w))*x=w \big\rangle.
\end{eqnarray}

We color $K_1$ and $K_2$ by the medial quandle $X$ of Lemma \ref{order five quandle}. Notice that $X$ contains $\{1, 2,3 \} $ and $\{4, 5 \}$ as trivial subquandles. If $f: Q(K_1) \to X$ is a map, then any choice of $f(a), f(b) \in \{1, 2, 3 \}$ or $f(a), f(b) \in \{4,5\}$ satisfies the relations in \eqref{eq prest 9}, and hence $f$ is a quandle homomorphism. On the other hand, if $f(a) \in   \{1, 2, 3 \}$ and $f(b) \in  \{4, 5 \}$ or $f(a) \in   \{4, 5 \}$ and $f(b) \in  \{1, 2, 3 \}$, then relations in  \eqref{eq prest 9} are not satisfied, and $f$ cannot be a quandle homomorphism. Similarly, if $f:Q(K_2) \to X$, then relations in \eqref{eq prest 11} imply that either $f(x), f(w) \in \{1, 2, 3 \}$ or $f(x), f(w) \in \{4,5\}$. Thus,  we obtain $|\Hom(Q(K_1), X)| = 13=|\Hom(Q(K_2), X)|$. In fact, we have  $\Hom(Q(K_1), X) \cong \Hom(Q(K_2), X)$ as medial quandles.
\par

Next, we color $Q(K_1)$ and $Q(K_2)$ by the quandle $\mathcal{I}(\mathbb{Z}_2[X])$ of idempotents, and claim that
$\Hom\big(Q(K_1), \mathcal{I}(\mathbb{Z}_2[X])\big) \not\cong \Hom\big(Q(K_2), \mathcal{I}(\mathbb{Z}_2[X])\big)$. For convenience, we write $$\mathcal{I}\big(\mathbb{Z}_2[X]\big)= \{u_1, u_2, u_3, u_4, u_5, u_6, u_7, u_8, u_9, u_{10}\},$$ where $u_i=e_i$ for $1 \le i \le 5$, $u_6= e_1+e_2+e_3$, $u_7= e_1+e_4+e_5$, $u_8= e_2+e_4+e_5$, $u_9= e_3+e_4+e_5$ and $u_{10}= e_1+e_2+e_3+e_4+e_5$. Then the multiplication table of the idempotent quandle is
$$
\mathcal{I}(\mathbb{Z}_2[X])=
\begin{pmatrix} 
u_1 & u_1 & u_1 & u_2 & u_2 & u_1 & u_1 & u_1 & u_1 & u_1\\
u_2 & u_2 & u_2 & u_3 & u_3 & u_2 & u_2 & u_2 & u_2 & u_2\\
u_3 & u_3 & u_3 & u_1 & u_1 & u_3 & u_3 & u_3 & u_3 & u_3\\
u_5 & u_5 & u_5 & u_4 & u_4 & u_5 & u_5 & u_5 & u_5 & u_5\\
u_4 & u_4 & u_4 & u_5 & u_5 & u_4 & u_4 & u_4 & u_4 & u_4\\
u_6 & u_6 & u_6 & u_6 & u_6 & u_6 & u_6 & u_6 & u_6 & u_6\\
u_7 & u_7 & u_7 & u_8 & u_8 & u_7 & u_7 & u_7 & u_7 & u_7\\
u_8 & u_8 & u_8 & u_9 & u_9 & u_8 & u_8 & u_8 & u_8 & u_8\\
u_9 & u_9 & u_9 & u_7 & u_7 & u_9 & u_9 & u_9 & u_9 & u_9\\
u_{10} & u_{10} & u_{10} & u_{10} & u_{10} & u_{10} & u_{10} & u_{10} & u_{10} & u_{10}\\
\end{pmatrix}.$$

Given a map $f: Q(K_1) \to \mathcal{I}(\mathbb{Z}_2[X])$, we have the following cases:
\begin{enumerate}
\item[(i)] $f(a), f(b) \in  \{ u_1, u_2, u_3, u_4, u_5 \}$.
\item[(ii)] $f(a), f(b) \in  \{ u_6, u_7, u_8, u_9, u_{10} \}$.
\item[(iii)] $f(a) \in   \{ u_1, u_2, u_3, u_4, u_5\}$ and $f(b) \in  \{ u_6, u_7, u_8, u_9, u_{10} \}$.
\item[(iv)] $f(a) \in  \{ u_6, u_7, u_8, u_9, u_{10} \}$ and $f(b) \in  \{ u_1, u_2, u_3, u_4, u_5\}$.
\end{enumerate}
Case (i) is equivalent to coloring by $X$, and hence the number of homomorphisms in this case is 13. Further, since $\{ u_6, u_7, u_8, u_9, u_{10} \}$ is a trivial quandle, the number of homomorphisms in Case (ii) is 25. Cases (iii) and (iv) work only when  $f(a) \in   \{u_1, u_2, u_3\}$ and $f(b) \in  \{ u_6, u_7, u_8, u_9, u_{10} \}$ or 
$f(a) \in  \{ u_6, u_7, u_8, u_9, u_{10} \}$ and $f(b) \in  \{ u_1, u_2, u_3\}$, which gives 30 additional homomorphisms. Thus, we have $|\Hom\big(Q(K_1), \mathcal{I}(\mathbb{Z}_2[X])\big)|= 68$. 
\par

Similarly, given a map $f: Q(K_2) \to \mathcal{I}(\mathbb{Z}_2[X])$, we have four cases:
\begin{enumerate}
\item[(i)] $f(x), f(w) \in  \{ u_1, u_2, u_3, u_4, u_5 \}$.
\item[(ii)] $f(x), f(w) \in  \{ u_6, u_7, u_8, u_9, u_{10} \}$.
\item[(iii)] $f(x) \in   \{ u_1, u_2, u_3, u_4, u_5\}$ and $f(w) \in  \{ u_6, u_7, u_8, u_9, u_{10} \}$.
\item[(iv)] $f(x) \in  \{ u_6, u_7, u_8, u_9, u_{10} \}$ and $f(w) \in  \{ u_1, u_2, u_3, u_4, u_5\}$.
\end{enumerate}
As for $Q(K_1)$, we see that the number of homomorphisms contributed by cases (i) and (ii) are 13 and 25, respectively. Case (iii) works when  $f(x) \in   \{u_1, u_2, u_3\}$ and $f(w) \in  \{ u_6, u_7, u_8, u_9, u_{10} \}$ or when $f(x) \in  \{ u_4, u_5 \}$ and $f(w) \in  \{ u_6, u_{10} \}$, which gives a total of $19$ homomorphisms. If $f(x) \in   \{u_4, u_5\}$ and $f(w) \in  \{u_7, u_8, u_9\}$, then $f$ fails to be a homomorphism. Furthermore, roles of $f(x)$ and $f(w)$ can be interchanged giving 19 additional homomorphisms. Thus, we have $|\Hom\big(Q(K_2), \mathcal{I}(\mathbb{Z}_2[X])\big)|= 76$, which proves our claim. 
\end{ex}
This shows that the enhancement is proper.
\end{proof}

\begin{rmk}
In Theorem \ref{idempotent enhancement} we can replace $\mathcal{I}(\mathbf{k}[X])$ by a maximal quandle in $\mathcal{I}(\mathbf{k}[X])$ containing all the trivial idempotents together with at least one non-trivial idempotent.
\end{rmk}
\medskip

\section{Computations of idempotents of quandle rings}\label{computer computations}
Idempotents of integral quandle rings of all quandles of order $3$ have been computed in \cite{BPS1}. Further, idempotents of integral quandle rings of certain quandles that are coverings or  (twisted) disjoint unions can be obtained from \cite[Sections 4 and 6]{ENSS2022}. In this section, we give the complete list of idempotents of integral as well as mod 2 quandle rings of all quandles of order less than six. We also determine the quandles for which the set of all idempotents is itself a quandle. The results have been verified using the Maple software \cite{Maple}.

%%%%%%%%%%%%%%%%%%%%%%%%%%%%%%%%%%%%%%%%%%%%%%%%%%%
%%  IDEMPOTENTS FOR QUANDLES OF ORDER 3    %%
%%%%%%%%%%%%%%%%%%%%%%%%%%%%%%%%%%%%%%%%%%%%%%%%%%%
{\tiny{
\begin{table}[H]
\caption{Idempotents for quandles of order  $3$.}
\label{Table2}
\begin{center}
\begin{tabular}{ |c|c|c|c|c|} 
 \hline
Quandle $X$  & $\mathcal{I}\big(\mathbb{Z}[X]\big)$ & Is $\mathcal{I}\big(\mathbb{Z}[X]\big)$ & $\mathcal{I}\big(\mathbb{Z}_2[X]\big)$ & Is $\mathcal{I}\big(\mathbb{Z}_2[X]\big)$ \\ 
&  & a quandle? &  & a quandle? \\  
\hline

$\left[ \begin{array}{c}

1 \;1 \;1  \\
2 \;2 \;2 \\
3 \;3\; 3 
\end{array} \right] $
& 
\makecell{ \\
    $\alpha e_1+\beta e_2+(1-\alpha-\beta) e_3.$\\
    \\
}
&
\makecell{ \\
	Yes.
}

&
\makecell{ \\
    $ e_1,~ e_2,~ e_3,$\\
$e_1+e_2+e_3.$  \\ \\
}
&
\makecell{ \\ Yes. }
 \\ 
\hline

$\left[ \begin{array}{c}
%\left[1, 1, 2\right]  \\
%\left[2, 2, 1\right] \\
%\left[3, 3, 3\right] 
1 \;1 \;2  \\
2 \;2 \;1 \\
3 \;3\; 3 
\end{array} \right] $
& 
\makecell{ \\
    $\alpha (e_1+ e_2)+(1-2\alpha) e_3$, \\
    $\alpha e_1+(1-\alpha) e_2.$\\
   \\
}
   &
   \makecell{ \\
   	No. 
   }
   
 &
\makecell{ \\
    $ e_1,~ e_2,~ e_3,$\\
$e_1+e_2+e_3.$  \\
}
&
\makecell{ \\ Yes. }
  \\
  \hline

$\left[ \begin{array}{c}
1 \;3 \;2  \\
3 \;2 \;1 \\
2 \;1\; 3 
\end{array} \right] $ 
& 
\makecell{\\
    $e_1,~e_2,~e_3$.\\
    }
    &
    \makecell{ \\
    	Yes.
    }

&
\makecell{ \\
    $ e_1,~ e_2,~ e_3,$\\
$e_1+e_2,~ e_2+e_3,~e_3+e_1$, \\
$e_1+e_2+e_3$. \\ \\}
&
\makecell{ \\ No. }
  \\ 
  \hline
\end{tabular}
\end{center}
\end{table}

%%%%%%%%%%%%%%%%%%%%%%%%%%%%%%%%%%%%%%%%%%%%%%%%%%%
%%  IDEMPOTENTS FOR QUANDLES OF ORDER 4    %%
%%%%%%%%%%%%%%%%%%%%%%%%%%%%%%%%%%%%%%%%%%%%%%%%%%%
\begin{table}[H]
\caption{Idempotents for quandles of order  $4$.}
\label{Table4}
\begin{center}
\begin{tabular}{ |c|c|c|c|c|} 
 \hline
Quandle $X$  & $\mathcal{I}\big(\mathbb{Z}[X]\big)$ & Is $\mathcal{I}\big(\mathbb{Z}[X]\big)$ & $\mathcal{I}\big(\mathbb{Z}_2[X]\big)$ & Is $\mathcal{I}\big(\mathbb{Z}_2[X]\big)$ \\ 
&  & a quandle? &  & a quandle? \\  
\hline

$\left[ \begin{array}{c}
1\; 1\; 1\; 1  \\
2\; 2\; 2\; 2 \\
3\; 3\; 3\; 3 \\ 
4\; 4\; 4\; 4
\end{array} \right] $ 
 & 
 \makecell{\\
    $\alpha e_1+\beta e_2+\gamma e_3+(1-\alpha-\beta-\gamma ) e_4$.\\
\\ 
}
&
\makecell{ \\
	Yes.
}

&
\makecell{\\
$e_1,~e_2,~e_3,~e_4,$\\
$e_1+e_2+e_3,~e_1+e_2+e_4$,\\
$e_1+e_3+e_4,~e_2+e_3+e_4$.\\ \\
}
&
\makecell{ \\ Yes.}
\\
\hline
%%%%%%%2

$\left[ \begin{array}{c}
1\; 1\; 1\; 1  \\
2\; 2\; 2\; 3 \\
3\; 3\; 3\; 2 \\ 
4\; 4\; 4\; 4
\end{array} \right] $
& 
\makecell{\\
$\alpha e_1+\beta (e_2+ e_3)+(1-\alpha-2\beta) e_4,$ \\
$(1-\alpha -\beta)  e_1+\alpha e_2 + \beta e_3.$\\ 
\\
       % Set of idempotents is not a quandle since the  right multiplication \\
        %by $\alpha e_1+(1-\alpha) e_4$ is not surjective for $\alpha \ne 0, 1$
    }
    &
    \makecell{ \\
    	No.
    }
&
\makecell{\\
$e_1,~e_2,~e_3,~e_4,$\\
$e_1+e_2+e_3,$\\
$e_2+e_3+e_4$.\\ \\
}
&
\makecell{ \\ Yes.}
\\
\hline

$\left[ \begin{array}{c}
1\; 1\; 1\; 2  \\
2\; 2\; 2\; 3 \\
3\; 3\; 3\; 1 \\ 
4\; 4\; 4\; 4
\end{array} \right] $ 
 & 
\makecell{\\
    $\alpha (e_1+e_2+ e_3)+(1-3\alpha) e_4,$ \\
    $\alpha e_1+\beta e_2+(1-\alpha-\beta) e_3.$\\
    \\
           % Set of idempotents is not a quandle since the right multiplication \\
           % by $\alpha (e_1+e_2+ e_3)+(1-3\alpha) e_4$ is not surjective $\alpha \ne 0$.
        }
     &
    \makecell{ \\
    	No.
    }
&
\makecell{\\ \\
$e_1,~e_2,~e_3,~e_4,$\\
$e_1+e_2+e_3$.\\ \\
}
&
\makecell{ \\ Yes.}
\\
\hline

$\left[ \begin{array}{c}
1\; 1\; 1\; 1  \\
2\; 2\; 4\; 3 \\
3\; 4\; 3\; 2 \\ 
4\; 3\; 2\; 4
\end{array} \right] $ 
& 
\makecell{\\
    $(1-3\alpha)e_1+\alpha(e_2+e_3+e_4),$ \\ 
    $(1-\alpha)e_1+\alpha e_4,$ \\
    $(1-\alpha)e_1+\alpha e_3,$ \\
    $(1-\alpha)e_1+\alpha e_2.$\\
\\
\\
           % Set of idempotents is not a quandle since the right distributivity\\
%fails for $2e_1-e_4$, $2e_1-e_3$ and $2e_1-e_2$.
}
 &
\makecell{ \\
	No.
}
&
\makecell{\\
$e_1,~e_2,~e_3,~e_4,$\\
$e_2+e_3,~e_2+e_4,$\\
$e_3+e_4,~e_2+e_3+e_4$.\\
}
&
\makecell{ \\ No.}
\\
\hline

$\left[ \begin{array}{c}
1\; 1\; 2\; 2  \\
2\; 2\; 1\; 1 \\
3\; 3\; 3\; 3 \\ 
4\; 4\; 4\; 4
\end{array} \right] $ 
 & 
\makecell{\\
    $\alpha(e_1+e_2)+\beta e_3+(1-2\alpha-\beta)e_4,$ \\ 
    $\alpha e_1+(1-\alpha)e_2+\beta(e_3-e_4).$ 
\\
\\
            %Set of idempotents is not a quandle since the right multiplication\\
%by $\alpha(e_1+e_2)+(1-2\alpha)e_4$ is not surjective for $\alpha \ne 0$.
}
 &
\makecell{ \\
	No.
}

&
\makecell{\\
$e_1,~e_2,~e_3,~e_4,$\\
$e_1+e_2+e_3,~e_1+e_2+e_4,$\\
$e_1+e_3+e_4,~e_2+e_3+e_4$.\\ \\
}
&
\makecell{ \\ Yes.}
\\
\hline

$\left[ \begin{array}{c}
1\; 1\; 2\; 2  \\
2\; 2\; 1\; 1 \\
4\; 4\; 3\; 3 \\ 
3\; 3\; 4\; 4
\end{array} \right] $
 & 
\makecell{\\
    $\alpha e_1+(1-\alpha)e_2,$
        $\alpha e_3+(1-\alpha) e_4.$ 
\\
  \\       %Set of idempotents is a quandle.  
 }
 &
\makecell{ \\
	Yes.}
&
\makecell{\\
$e_1,~e_2,~e_3,~e_4,$\\
$e_1+e_2+e_3,~e_1+e_2+e_4,$\\
$e_1+e_3+e_4,~e_2+e_3+e_4$.\\ \\
}
&
\makecell{ \\ Yes.}
\\
\hline
%%%%%%7

$\left[ \begin{array}{c}
1\; 4\; 2\; 3  \\
3\; 2\; 4\; 1 \\
4\; 1\; 3\; 2 \\ 
2\; 3\; 1\; 4
\end{array} \right] $ 
& 
\makecell{\\
    $e_1,~e_2,~e_3,~e_4$.
\\
  \\       %Set of idempotents is a quandle.   
}
 &
\makecell{ \\
	Yes.
}
&
\makecell{\\
$e_1,~e_2,~e_3,~e_4.$\\ \\}
&
\makecell{ \\ Yes.}
\\
  \hline
\end{tabular}
\end{center}
\end{table}

%%%%%%%%%%%%%%%%%%%%%%%%%%%%%%%%%%%%%%%%%%%%%%%%%%%
%%  IDEMPOTENTS OF THE FIRST 11 QUANDLES OF ORDER 5    %%
%%%%%%%%%%%%%%%%%%%%%%%%%%%%%%%%%%%%%%%%%%%%%%%%%%%

\begin{table}[H]
\caption{Idempotents for quandles of order  $5$.}
\label{Table4}
\begin{center}
\begin{tabular}{ |c|c|c|c|c|} 
 \hline
Quandle $X$  & $\mathcal{I}\big(\mathbb{Z}[X]\big)$ & Is $\mathcal{I}\big(\mathbb{Z}[X]\big)$ & $\mathcal{I}\big(\mathbb{Z}_2[X]\big)$ & Is $\mathcal{I}\big(\mathbb{Z}_2[X]\big)$ \\ 
&  & a quandle? &  & a quandle? \\  %\\
\hline

$\left[ \begin{array}{c}
1\; 1\; 1\; 1\; 1  \\
2\; 2\; 2\; 2\; 2 \\
3\; 3\; 3\; 3\; 3 \\ 
4\; 4\; 4\; 4\;4 \\
5\; 5\; 5\; 5\; 5
\end{array} \right] $
 & 
\makecell{
    $\alpha e_1+\beta e_2+\gamma e_3+\delta e_4+$\\
    $(1-\alpha-\beta-\gamma-\delta )e_5$. 
\\
  \\       %Set of idempotents is a quandle.   
} 
&
\makecell{ \\
	Yes.}   
&
\makecell{\\
 $e_1,~e_2,~e_3,~e_4,~e_5$,\\
	$e_1+e_2+e_3,~e_1+e_2+e_4$,\\
	$e_1+e_2+e_5,~e_1+e_3+e_4$,\\
	$e_1+e_3+e_5,~e_1+e_4+e_5$,\\
	$e_2+e_3+e_4,~e_2+e_3+e_5$,\\
	$e_2+e_4+e_5,~e_3+e_4+e_5$,\\
	$e_1+e_2+e_3+e_4+e_5$.\\ \\
	}  
	&
\makecell{ \\
	Yes.}  
\\
\hline
%%%%%%%2

$\left[ \begin{array}{c}
1\; 1\; 1\; 1\; 1  \\
2\; 2\; 2\; 2\; 2 \\
3\; 3\; 3\; 3\; 4 \\ 
4\; 4\; 4\; 4\;3 \\
5\; 5\; 5\; 5\; 5
\end{array} \right] $ 
& 
\makecell{\\
    $(1-\alpha-2\beta-\gamma) e_1+\alpha e_2+\beta(e_3+e_4)+\gamma e_5,$ \\ 
    $(1- \alpha-\beta-\gamma) e_1+\alpha e_2+\beta  e_3+\gamma e_4.$
\\
\\
           % Set of idempotents is not a quandle since the right multiplication\\
%by $(1-\gamma)e_1 +\gamma e_5$ is not surjective for $\gamma \ne 0,1$.
}
&
\makecell{ \\
	No.} 
&
\makecell{\\
$e_1,~e_2,~e_3,~e_4,~e_5$,\\
	$e_1+e_2+e_3,~e_1+e_2+e_4$,\\
	$e_1+e_2+e_5,~e_1+e_3+e_4$,\\
	$e_2+e_3+e_4,~e_3+e_4+e_5$,\\
	$e_1+e_2+e_3+e_4+e_5$.\\ \\
	}  
	&
\makecell{ \\
	Yes. }  
\\
\hline
%%%%%%%3

$\left[ \begin{array}{c}
1\; 1\; 1\; 1\; 1  \\
2\; 2\; 2\; 2\; 3 \\
3\; 3\; 3\; 3\; 4 \\ 
4\; 4\; 4\; 4\;2 \\
5\; 5\; 5\; 5\; 5
\end{array} \right] $
 & 
\makecell{\\
    $(1-3\alpha-\beta) e_1+\alpha (e_2+ e_3+e_4)+\beta e_5,$ \\ 
    $(1- \alpha-\beta-\gamma) e_1+\alpha e_2+\beta  e_3+\gamma e_4.$
\\
\\
           % Set of idempotents is not a quandle since the right multiplication\\
%by $(1-\beta)e_1 +\beta e_5$ is not surjective for $\beta \ne 0,1$.
}
&
\makecell{ \\
	No.} 
	&
\makecell{ \\$e_1,~e_2,~e_3,~e_4,~e_5$,\\
	$e_1+e_2+e_3,~e_1+e_2+e_4$,\\
	$e_1+e_3+e_4,~e_2+e_3+e_4$,\\
		$e_1+e_2+e_3+e_4+e_5$.\\ \\
	}  
	&
\makecell{ \\
	Yes. }  
\\ 
\hline
%%%%%%%4
$\left[ \begin{array}{c}
1\; 1\; 1\; 1\; 2  \\
2\; 2\; 2\; 2\; 1 \\
3\; 3\; 3\; 3\; 4 \\ 
4\; 4\; 4\; 4\;3 \\
5\; 5\; 5\; 5\; 5
\end{array} \right] $
& 
\makecell{ \\
    $\alpha(e_1+e_2)+\beta(e_3+e_4)+(1-2\alpha-2\beta)e_5$, \\
        $(1- \alpha-\beta-\gamma) e_1+\alpha e_2+\beta  e_3+\gamma e_4.$ 
}
&
\makecell{ \\
	No.} 
		&
\makecell{ \\$e_1,~e_2,~e_3,~e_4,~e_5$,\\
	$e_1+e_2+e_3,~e_1+e_2+e_4$,\\
	$e_1+e_2+e_5,~e_1+e_3+e_4$,\\
	$e_2+e_3+e_4,~e_3+e_4+e_5$,\\
	$e_1+e_2+e_3+e_4+e_5$.\\ \\
	}  
	&
\makecell{ \\
	Yes. }  
\\ 
\hline
%%%%%%%5
$\left[ \begin{array}{c}
1\; 1\; 1\; 1\; 2  \\
2\; 2\; 2\; 2\; 3 \\
3\; 3\; 3\; 3\; 4 \\ 
4\; 4\; 4\; 4\;1 \\
5\; 5\; 5\; 5\; 5
\end{array} \right] $
& 
\makecell{\\
    $\alpha(e_1+e_2+ e_3+e_4)+(1-4\alpha)e_5$, \\ 
    $(1- \alpha-\beta-\gamma) e_1+\alpha e_2+\beta  e_3+\gamma e_4.$
\\
\\
            %Set of idempotents is not a quandle since the right multiplication\\
%by $\alpha(e_1+e_2+ e_3+e_4)+(1-4\alpha)e_5$ is not surjective for $\alpha \ne 0$.
}
&
\makecell{ \\
	No.} 
		&
\makecell{ \\$e_1,~e_2,~e_3,~e_4,~e_5$,\\
	$e_1+e_2+e_3,~e_1+e_2+e_4$,\\
	$e_1+e_3+e_4,~e_2+e_3+e_4$,\\
	$e_1+e_2+e_3+e_4+e_5$.\\ \\
	}  
	&
\makecell{ \\
	Yes. }  
\\
\hline
%%%%%%%6
$\left[ \begin{array}{c}
1\; 1\; 1\; 1\; 1  \\
2\; 2\; 2\; 2\; 2 \\
3\; 3\; 3\; 5\; 4 \\ 
4\; 4\; 5\; 4\;3 \\
5\; 5\; 4\; 3\; 5
\end{array} \right] $
& 
\makecell{\\
    $\alpha e_1+(1-\alpha-3\beta) e_2+\beta(e_3+e_4+e_5)$, \\
    $\alpha e_1+(1-\alpha-\beta)e_2+ \beta e_3$, \\
    $\alpha e_1+(1-\alpha-\beta)e_2+ \beta e_4$, \\
    $\alpha e_1+(1-\alpha-\beta)e_2+ \beta e_5.$
\\
\\
           % Set of idempotents is not a quandle since the right distributivity\\
%fails for $2 e_1- e_3,$    $2 e_1- e_4$ and    $2 e_1- e_5.$
}
&
\makecell{ \\
	No.} 
		&
\makecell{ $e_1,~e_2,~e_3,~e_4,~e_5,~e_3+e_4$,\\
		$e_3+e_5,~e_4+e_5,~e_1+e_2+e_3$,\\
	$~e_1+e_2+e_4,~e_1+e_2+e_5$,\\
	$e_1+e_2+e_3+e_4+e_5$.\\
	}  
	&
\makecell{ \\
	No. }  
\\ 
\hline
%%%%%%7
$\left[ \begin{array}{c}
1\; 1\; 1\; 1\; 1  \\
2\; 2\; 2\; 3\; 3 \\
3\; 3\; 3\; 2\; 2 \\ 
4\; 4\; 4\; 4\;4 \\
5\; 5\; 5\; 5\; 5
\end{array} \right] $
& 
\makecell{\\
    $(1-\alpha-\beta)e_1+\alpha e_2+\beta e_3+\gamma(e_4-e_5)$, \\
    $(1-2\alpha-\beta-\gamma)e_1+\alpha (e_2+e_3)+\beta e_4+\gamma e_5.$
\\
\\
           % Set of idempotents is not a quandle since the right multiplication\\
%by $(1-\beta)e_1+\beta e_4$ is not surjective for $\beta \ne 0, 1$.
}
&
\makecell{ \\
	No.}
		&
\makecell{ \\$e_1,~e_2,~e_3,~e_4,~e_5$,\\
	$e_1+e_2+e_3,~e_1+e_4+e_5$,\\
	$e_2+e_3+e_4,~e_2+e_3+e_5$,\\
	$e_2+e_4+e_5,~e_3+e_4+e_5$,\\
	$e_1+e_2+e_3+e_4+e_5$.\\ \\
	}  
	&
\makecell{ \\
	Yes. }   
\\
  \hline
  %%%%%%%8
$\left[ \begin{array}{c}
1\; 1\; 1\; 1\; 1  \\
2\; 2\; 2\; 3\; 3 \\
3\; 3\; 3\; 2\; 2 \\ 
4\; 5\; 5\; 4\;4 \\
5\; 4\; 4\; 5\; 5
\end{array} \right] $
& 
\makecell{\\
$e_1+\alpha(e_2- e_3)+\beta(e_4-e_5)$, \\
 $(1-\alpha-\beta)e_1+\alpha e_4+\beta e_5$, \\
$(1-\alpha-\beta)e_1+\alpha e_2+\beta e_3$, \\ 
$(1-2\alpha-2\beta)e_1+\alpha (e_2+e_3)+\beta(e_4+e_5).$
\\
\\
           % Set of idempotents is not a quandle since the right multiplication\\
%by $(1-2\beta)e_1+ \beta(e_4+e_5)$ is not surjective for $\beta \ne 0$.
}
&
\makecell{ \\
	No.} 
		&
\makecell{ $e_1,~e_2,~e_3,~e_4,~e_5$,\\
	$e_1+e_2+e_3,~e_1+e_4+e_5$,\\
	$e_2+e_3+e_4,~e_2+e_3+e_5$,\\
	$e_2+e_4+e_5,~e_3+e_4+e_5$,\\
	$e_1+e_2+e_3+e_4+e_5$.\\ 
	}  
	&
\makecell{ \\
	Yes. }  
\\ 
\hline
\end{tabular}
\end{center}
\end{table}

%%%%%%%%%%%%%%%%%%%%%%%%%%%%%%%%%%%%%%%%%%%%%%%%%%%
%%  IDEMPOTENTS OF THE FIRST OF ORDER 5 (from 12 to 22)    %%
%%%%%%%%%%%%%%%%%%%%%%%%%%%%%%%%%%%%%%%%%%%%%%%%%%%

\begin{table}[H]
\caption{Idempotents for quandles of order  $5$.}
\label{Table4}
\begin{center}
\begin{tabular}{ |c|c|c|c|c|} 
	\hline
Quandle $X$  & $\mathcal{I}\big(\mathbb{Z}[X]\big)$ & Is $\mathcal{I}\big(\mathbb{Z}[X]\big)$ & $\mathcal{I}\big(\mathbb{Z}_2[X]\big)$ & Is $\mathcal{I}\big(\mathbb{Z}_2[X]\big)$ \\ 
&  & a quandle? &  & a quandle? \\ % \\
\hline
%%%%%%9
$\left[ \begin{array}{c}
1\; 1\; 1\; 2\; 2  \\
2\; 2\; 2\; 1\; 1 \\
3\; 3\; 3\; 3\; 3 \\ 
4\; 4\; 5\; 4\;4 \\
5\; 5\; 4\; 5\; 5
\end{array} \right] $
& 
\makecell{
    $\alpha(e_1+e_2)+\beta e_4 +(1-2\alpha-\beta)e_5$, \\
    $\alpha(e_1+e_2)+(1-2\alpha-2\beta)e_3+\beta(e_4+e_5)$, \\
    $\alpha e_1+\beta e_2 +(1-\alpha-\beta)e_3$, \\ 
    $\alpha e_1+(1-\alpha)e_2+\beta(e_4-e_5)$,
\\
\\
           % Set of idempotents is not a quandle since the right multiplication\\
%by $e_1+e_2+ e_4 -2e_5$ is not surjective.
}
&
\makecell{No.}
	&
\makecell{ \\
$e_1,~e_2,~e_3,~e_4,~e_5$,\\
	$e_1+e_2+e_3,~e_1+e_2+e_4$,\\
	$e_1+e_2+e_5,~e_1+e_4+e_5$,\\
	$e_2+e_4+e_5,~e_3+e_4+e_5$,\\
	$e_1+e_2+e_3+e_4+e_5$.\\ \\
	}  
	&
\makecell{ \\
	Yes. }  
\\ 
  \hline
%%%%%%%10
$\left[ \begin{array}{c}
1\; 1\; 1\; 2\; 2  \\
2\; 2\; 2\; 3\; 3 \\
3\; 3\; 3\; 1\; 1 \\ 
4\; 4\; 4\; 4\;4 \\
5\; 5\; 5\; 5\; 5
\end{array} \right] $
&  
\makecell{
    $\alpha (e_1+e_2+e_3)+(1-3\alpha-\beta)e_4+\beta e_5$, \\ 
    $(1-\alpha-\beta)e_1+\alpha e_2+ \beta e_3 +\gamma (e_4-e_5).$
\\
\\
           % Set of idempotents is not a quandle since the right multiplication\\
%by $e_1+e_2+ e_3 -2e_4$ is not surjective.
}
&
\makecell{No.}
	&
\makecell{ \\$e_1,~e_2,~e_3,~e_4,~e_5$,\\
	$e_1+e_2+e_3,~e_1+e_4+e_5$,\\
	$e_2+e_4+e_5,~e_3+e_4+e_5$,\\
	$e_1+e_2+e_3+e_4+e_5$.\\ \\
	}  
	&
\makecell{ \\
	Yes. }  
\\ 
\hline
%%%%%%11
$\left[ \begin{array}{c}
1\; 1\; 1\; 2\; 3  \\
2\; 2\; 2\; 3\; 1 \\
3\; 3\; 3\; 1\; 2 \\ 
4\; 4\; 4\; 4\;4 \\
5\; 5\; 5\; 5\; 5
\end{array} \right] $
& 
\makecell{\\
    $\alpha(e_1+e_2 + e_3)+(1-3\alpha)e_4$, \\ 
    $\alpha(e_1+e_2 + e_3)+(1-3\alpha)e_5$, \\ 
    $\alpha e_1+\beta e_2+(1-\alpha -\beta)e_3.$\\
\\
\\
           % Set of idempotents is not a quandle since the right multiplication\\
%by $e_1+e_2+ e_3 -2e_4$ is not surjective.
}
&
\makecell{No.} 
	&
\makecell{ \\
$e_1,~e_2,~e_3,~e_4,~e_5$,\\	
$e_1+e_2+e_3$,\\
$e_1+e_2+e_3+e_4+e_5$.\\ \\
	}  
	&
\makecell{ \\
	No. }  
\\
  \hline
%%%%%%% 12
 
$\left[ \begin{array}{c}
1\; 1\; 1\; 1\; 1\\
2\; 2\; 2\; 2\; 2\\
3\; 3\; 3\; 3\; 3\\
5\; 5\; 5\; 4\; 4\\
4\; 4\; 4\; 5\; 5
\end{array} \right] $
& 
\makecell{
    $-(\alpha+\beta)e_1+\alpha e_2+\beta e_3+(1-\gamma)e_4+\gamma e_5$, \\ 
    $(1-\alpha-\beta -2\gamma)e_1+\alpha e_2+\beta e_3+\gamma(e_4+e_5).$
\\
\\
          %  Set of idempotents is not a quandle since the right multiplication\\
%by $-e_1+e_4+e_5$ is not surjective.
}
&
\makecell{No.} 
	&
\makecell{ \\
	$e_1,~e_2,~e_3,~e_4,~e_5$,\\
	$e_1+e_2+e_3,~e_1+e_2+e_4$,\\
	$e_1+e_2+e_5,~e_1+e_3+e_4$,\\
	$e_1+e_3+e_5,~e_1+e_4+e_5$,\\
	$e_2+e_3+e_4,~e_2+e_3+e_5$,\\
	$e_2+e_4+e_5,~e_3+e_4+e_5$,\\
	$e_1+e_2+e_3+e_4+e_5$.\\ \\
	}  
	&
\makecell{ \\
	Yes. }  
\\ 
\hline
%%%%%%% 13
$\left[ \begin{array}{c}
1\; 1\; 1\; 1\; 1\\
2\; 2\; 2\; 3\; 3\\
3\; 3\; 3\; 2\; 2\\
5\; 5\; 5\; 4\; 4\\
4\; 4\; 4\; 5\; 5
\end{array} \right] $
&  
\makecell{
    $-2\alpha e_1+\alpha(e_2+e_3)+\beta e_4+(1-\beta)e_5$, \\ 
    $(1-\alpha -\beta)e_1+\alpha e_2+\beta e_3$, \\ 
    $(1-2\alpha-2\beta)e_1+\alpha(e_2+e_3)+\beta(e_4+e_5).$
\\
\\
           % Set of idempotents is not a quandle since the right multiplication\\
%by $-3e_1+e_2+e_3+e_4+e_5$ is not surjective.
}
&
\makecell{No.}  
	&
\makecell{ \\	
$e_1,~e_2,~e_3,~e_4,~e_5$,\\
$e_1+e_2+e_3,~e_1+e_4+e_5$,\\
$e_2+e_3+e_4,~e_2+e_3+e_5$,\\
$e_2+e_4+e_5,~e_3+e_4+e_5$,\\
$e_1+e_2+e_3+e_4+e_5$.\\ \\
	}  
	&
\makecell{ \\
	Yes.}  
\\
\hline
%%%%%%% 14
$\left[ \begin{array}{c}
1\; 1\; 1\; 2\; 2\\
2\; 2\; 2\; 3\; 3\\
3\; 3\; 3\; 1\; 1\\
5\; 5\; 5\; 4\; 4\\
4\;  4\;  4\;  5\;  5
\end{array} \right] $
& 
\makecell{
$(1-\alpha -\beta)e_1+\alpha e_2+\beta e_3$,\\
    $\alpha e_4+(1-\alpha)e_5$, \\ 
    $-(1+2\alpha)(e_1+e_2+e_3)+(2+3\alpha)(e_4+e_5).$
\\
\\
          %  Set of idempotents is not a quandle since the right multiplication\\
%by $-e_1-e_2-e_3+2e_4+2e_5$ is not surjective.
}
&
\makecell{No.} 
	&
\makecell{ \\$e_1,~e_2,~e_3,~e_4,~e_5$,\\
	$e_1+e_2+e_3,~e_1+e_4+e_5$,\\
	$e_2+e_4+e_5,~e_3+e_4+e_5$,\\
		$e_1+e_2+e_3+e_4+e_5$.\\ \\
	}  
	&
\makecell{ \\
	Yes. }  
\\ 
\hline
%%%%%%% 15
$\left[ \begin{array}{c}
1\; 1\; 1\; 1\; 1\\
2\; 2\; 5\; 3\; 4\\
3\; 4\; 3\; 5\; 2\\
4\; 5\; 2\; 4\; 3\\
5\; 3\; 4\; 2\; 5
\end{array} \right] $ 
& 
\makecell{\\
    $e_1+\alpha(e_2+e_3-e_4-e_5)$, \\ 
    $e_1+\alpha(e_2-e_3+e_4-e_5)$, \\
    $e_1+\alpha(e_2-e_3-e_4+e_5)$, \\  
    $\alpha e_1+(1-\alpha)e_2$, \\ 
    $\alpha e_1+(1-\alpha)e_3$, \\ 
    $\alpha e_1+(1-\alpha)e_4$, \\ 
    $\alpha e_1+(1-\alpha)e_5$, \\ 
    $(1-4\alpha)e_1+\alpha(e_2+e_3+e_4+e_5).$
\\
\\
          %  Set of idempotents is not a quandle since the right distributivity\\
%fails for $2e_1-e_2$, $2e_1-e_3$ and  $2e_1-e_4$.
}
&
\makecell{No.}   
	&
\makecell{  $e_1,~e_2,~e_3,~e_4,~e_5$,\\
	$e_1+e_2+e_3+e_4+e_5$.\\ 
	}  
	&
\makecell{ \\
	No.}  
\\ 
\hline
%%%%%%% 16
$\left[ \begin{array}{c}
1\; 1\; 2\; 2\; 2\\
2\; 2\; 1\; 1\; 1\\
3\; 3\; 3\; 3\; 4\\
4\; 4\; 4\; 4\; 3\\
5\; 5\; 5\; 5\; 5
\end{array} \right] $
& 
\makecell{
    $\alpha(e_1+e_2)+\beta(e_3+e_4)+(1-2\alpha-2\beta)e_5$, \\ 
    $\alpha(e_1+e_2)+\beta e_3+(1-2\alpha-\beta)e_4$, \\ 
    $\alpha e_1+(1-\alpha)e_2+\beta(e_3+e_4-2e_5)$, \\ 
    $\alpha e_1+(1-\alpha)e_2+\beta(e_3 -e_4).$
\\
\\
            %Set of idempotents is not a quandle since the right multiplication\\
%by $e_1+e_2+e_3-2e_4$ is not surjective.
}
&
\makecell{No.}
	&
\makecell{ \\	$e_1,~e_2,~e_3,~e_4,~e_5$,\\
	$e_1+e_2+e_3,~e_1+e_2+e_4$,\\
	$e_1+e_2+e_5,~e_1+e_3+e_4$,\\
	$e_2+e_3+e_4,~e_3+e_4+e_5$,\\
	$e_1+e_2+e_3+e_4+e_5$.\\ \\
	}  
	&
\makecell{ \\
	Yes.}  
\\ 
\hline
\end{tabular}
\end{center}
\end{table}

%%%%%%%%%%%%%%%%%%%%%%%%%%%%%%%%%%%%
%%%%%%%%%%%%%%%%%%%%%%%%%%%%%%%%%%%%

\begin{table}
\caption{Idempotents for quandles of order  $5$.}
\label{Table4}
\begin{center}
\begin{tabular}{ |c|c|c|c|c|} 
	\hline
Quandle $X$  & $\mathcal{I}\big(\mathbb{Z}[X]\big)$ & Is $\mathcal{I}\big(\mathbb{Z}[X]\big)$ & $\mathcal{I}\big(\mathbb{Z}_2[X]\big)$ & Is $\mathcal{I}\big(\mathbb{Z}_2[X]\big)$ \\ 
&  & a quandle? &  & a quandle? \\ 
\hline
%%%%%%% 17
$\left[ \begin{array}{c}
1\; 1\; 2\; 2\; 2\\
2\; 2\; 1\; 1\; 1\\
3\; 3\; 3\; 5\; 4\\
4\; 4\; 5\; 4\; 3\\
5\; 5\; 4\; 3\; 5
\end{array} \right] $
& 
\makecell{\\
    $\alpha(e_1+e_2)+(1-2\alpha)e_3$, \\ 
    $\alpha(e_1+e_2)+(1-2\alpha)e_4$, \\ 
    $\alpha(e_1+e_2)+(1-2\alpha)e_5$, \\ 
    $\alpha e_1+(1-\alpha)e_2$, \\
    $(2+3\alpha)(e_1+e_2)-(1+2\alpha)(e_3+e_4+e_5).$
\\
\\
          %  Set of idempotents is not a quandle since the right multiplication\\
%by $e_1+e_2-e_3$ is not surjective.
}
&
\makecell{No.}
	&
\makecell{ $e_1,~e_2,~e_3,~e_4,~e_5$,\\
	$e_1+e_2,~e_3+e_5,~e_4+e_5$,\\
	$e_1+e_2+e_3,~e_1+e_2+e_4$,\\
	$e_1+e_2+e_5,~e_3+e_4+e_5$,\\
	$e_1+e_2+e_3+e_4+e_5$.\\ 
	}  
	&
\makecell{ 	
	No. }  
\\
 \hline
%%%%%% 18
$\left[ \begin{array}{c}
1\; 1\; 2\; 2\; 2\\
2\; 2\; 1\; 1\; 1\\
3\; 3\; 3\; 3\; 3\\
5\; 5\; 5\; 4\; 4\\
4\; 4\; 4\; 5\; 5
\end{array} \right] $ 
& 
\makecell{
    $\alpha(e_1+e_2)-2\alpha e_3+\beta e_4+(1-\beta)e_5$, \\
    $\alpha(e_1+e_2)+(1-2\alpha-2\beta) e_3+\beta(e_4+e_5)$,\\
    $\alpha e_1+(1-\alpha) e_2-e_3+\beta e_4+(1-\beta)e_5$, \\ 
    $\alpha e_1+(1-\alpha) e_2-2\beta e_3+\beta (e_4+e_5)$.
    \\
    \\
%Set of idempotents is not a quandle since the right multiplication\\
%by $e_1+e_2-2e_3+e_4$ is not surjective.
}
&
\makecell{No.}
	&
\makecell{\\ $e_1,~e_2,~e_3,~e_4,~e_5$,\\
	$e_1+e_2+e_3,~e_1+e_2+e_4$,\\
	$e_1+e_2+e_5,~e_1+e_3+e_4$\\
	$e_1+e_3+e_5,~e_1+e_4+e_5$,\\
	$e_2+e_3+e_4,~e_2+e_3+e_5$,\\
	$e_2+e_4+e_5,~e_3+e_4+e_5$,\\
	$e_1+e_2+e_3+e_4+e_5$.\\ \\
	}  
	&
\makecell{ \\
	Yes. }  
\\
  \hline
 %%%%%%% 19
$\left[ \begin{array}{c}
1\; 1\; 2\; 2\; 2\\
2\; 2\; 1\; 1\; 1\\
4\; 5\; 3\; 5\; 4\\
5\; 3\; 5\; 4\; 3\\
3\; 4\; 4\; 3\; 5
\end{array} \right] $
& 
\makecell{\\
    $e_3+\alpha(e_1+e_2-e_4-e_5)$, \\ 
    $e_4+\alpha(e_1+e_2-e_3-e_5)$, \\ 
    $e_5+\alpha(e_1+e_2-e_3-e_4)$, \\ 
    $\alpha e_1+(1-\alpha)e_2$,\\ 
    $(2+3\alpha)(e_1+e_2)-(1+2\alpha)(e_3+e_4+e_5).$
\\
\\
           % Set of idempotents is not a quandle since the right distributivity\\
%fails for $e_1+e_2+ e_3-e_4-e_5$, $e_1+e_2- e_3+e_4-e_5$ and $e_1+e_2- e_3-e_4+e_5$.
}
&
\makecell{No.}
	&
\makecell{ $e_1,~e_2,~e_3,~e_4,~e_5,~e_3+e_4$,\\
	$e_3+e_5,~e_4+e_5,~e_3+e_4+e_5$,\\
	$e_1+e_2+e_3+e_4+e_5$.\\ 
	}  
	&
\makecell{ \\
	No.}   
\\
\hline
%%%%%% 20
$\left[ \begin{array}{c}
1\; 3\; 4\; 5\; 2\\
3\; 2\; 5\; 1\; 4\\
4\; 5\; 3\; 2\; 1\\
5\; 1\; 2\; 4\; 3\\
2\; 4\; 1\; 3\; 5
\end{array} \right] $
&   
\makecell{
    $e_1$, 
    $e_2$, 
    $e_3$, 
    $e_4$, 
    $e_5$.
\\
\\
         %Set of idempotents is a quandle.   
}
&
\makecell{Yes.}

	&
\makecell{\\
 $e_1,~e_2,~e_3,~e_4,~e_5$,\\
$e_1+e_4,~e_1+e_5,~e_2+e_3$,\\
$e_2+e_4,~e_2+e_5,~e_3+e_4$,\\
$e_3+e_5,~e_1+e_2+e_5$,\\
$e_1+e_3+e_4,~e_1+e_3+e_5$,\\
		$e_1+e_4+e_5,~e_2+e_3+e_4$,\\
		$e_2+e_3+e_5,~e_2+e_4+e_5$,\\
		$e_3+e_4+e_5,~e_1+e_2+e_3+e_4$,\\
		$e_1+e_2+e_3+e_5,~e_1+e_2+e_4+e_5$,\\
		$e_1+e_3+e_4+e_5,~e_2+e_3+e_4+e_5$,\\
		$e_1+e_2+e_3+e_4+e_5$.\\ \\
	}  
	&
\makecell{ \\
	No.}  
\\
  \hline
  %%%%%%% 21
$\left[ \begin{array}{c}
1\; 4\; 5\; 3\; 2\\
3\; 2\; 4\; 5\; 1\\
2\; 5\; 3\; 1\; 4\\
5\; 1\; 2\; 4\; 3\\
4\; 3\; 1\; 2\; 5
\end{array} \right] $
&   
\makecell{
    $e_1$,
    $e_2$, 
    $e_3$, 
    $e_4$, 
    $e_5$.
\\
\\
         %Set of idempotents is a quandle.   
    }
&
\makecell{Yes.}
	&
\makecell{ \\ $e_1,~e_2,~e_3,~e_4,~e_5,~e_1+e_2+e_3+e_4$,\\
	$e_1+e_2+e_3+e_5,~e_1+e_2+e_4+e_5$,\\
	$e_1+e_3+e_4+e_5,~e_2+e_3+e_4+e_5$,\\
	$e_1+e_2+e_3+e_4+e_5$.\\ \\
	}  
	&
\makecell{ \\
	No.}  
\\
\hline
%%%%%% 22
$\left[ \begin{array}{c}
1\; 4\; 5\; 2\; 3\\
3\; 2\; 1\; 5\; 4\\
4\; 5\; 3\; 1\; 2\\
5\; 3\; 2\; 4\; 1\\
2\; 1\; 4\; 3\; 5
\end{array} \right] $
&   
\makecell{
    $e_1$,
    $e_2$, 
    $e_3$, 
    $e_4$, 
    $e_5$.
\\
\\
          %Set of idempotents is a quandle.   
     }
 &
\makecell{Yes.}
	&
\makecell{ \\ $e_1,~e_2,~e_3,~e_4,~e_5,~e_1+e_2+e_3+e_4$,\\
	$e_1+e_2+e_3+e_5,~e_1+e_2+e_4+e_5$,\\
	$e_1+e_3+e_4+e_5,~e_2+e_3+e_4+e_5$,\\
	$e_1+e_2+e_3+e_4+e_5$.\\ \\
	}  
	&
\makecell{ \\
	No.}  
\\
\hline
\end{tabular}
\end{center}
\end{table}}}
 
 \newpage
\begin{rmk}
The following observations are apparent from the preceding tables:
\begin{enumerate}
\item It has been conjectured in \cite[Conjecture 3.10]{ENSS2022} that the integral quandle ring of a finite latin quandle has only trivial idempotents. The data shows that the conjecture holds for all quandles of order less than six.
\item For the integral quandle ring of a quandle of order less than six, the family of idempotents given by a {\it linear expression of basis elements} forms a trivial quandle with respect to the ring multiplication. Thus, each infinite set of idempotents decomposes as a union of infinite trivial quandles.
\end{enumerate}

\end{rmk}

The  tables also show that all the idempotents of integral quandle rings have augmentation value 1. On the other hand, Proposition \ref{prop commutative quandle}, preceding tables and Proposition \ref{complex dihedral idempotents} show that idempotents of a mod 2 and that of a complex quandle ring can have augmentation value 0 and 1 both. In general, the following seems to be the case for integral quandle rings.

\begin{conj}\label{augmentation one conjecture}
Any non-zero idempotent of the integral quandle ring of a quandle has augmentation value 1.
\end{conj}
 \medskip

\section{Peirce spectra of complex quandle algebras}\label{sec Peirce spectra}
In this concluding section, we discuss Peirce spectrum of complex quandle algebras of finite rank. In associative algebras, the operator induced by multiplication by an idempotent is a projection onto a subspace, and hence has eigenvalues 1 and 0. For non-associative algebras, the eigenvalues of the operator induced by an idempotent can be arbitrary in general.  Given an idempotent $u$ of a non-associative algebra $A$ over $\mathbb{C}$, let $\sigma(u)$ denote the Peirce spectrum of $u$, which is the set of all eigenvalues of the operator induced by $u$. The Peirce spectrum of an idempotent $u$ induces the Peirce decomposition of the algebra $A$, which is the decomposition of $A$ into a direct sum of corresponding eigenspaces. The Peirce spectrum $\sigma(A)$ of the algebra $A$ is the set of all possible distinct eigenvalues in $\sigma(u)$ as $u$ runs over all idempotents of $A$. Though challenging to compute in general, the Peirce spectrum is well-understood for some classes of algebras. For example, if $A$ is a power-associative algebra, then $\sigma(A)=\{0, \frac{1}{2}, 1\}$ \cite{MR0210757}. Similarly, if $A$ is a Hsiang algebra (which appears in the classification of cubic minimal cones \cite[Chapter 6]{MR3243534}), then $\sigma(A) = \{1, -1, \frac{1}{2}, -\frac{1}{2} \}$. See \cite{MR4191467} for a nice survey about idempotents of non-associative algebras.
\par

Note that, quandle algebras are non-commutative in general. Since our quandles are right-distributive, we consider right multiplication operators induced by idempotents. The quandle algebra of a trivial quandle is associative and it follows from the description of its idempotents (see \eqref{idempotents trivial quandle}) that its Peirce spectrum is simply $\{ 1\}$. On the other hand, quandle algebra of a non-trivial quandle is non-associative, in fact, it is not power-associative \cite{MR3915329}. This, together with the fact that quandle algebras are generated by idempotents, makes it interesting to explore their Peirce spectra.
\par

Given a quandle $X$ and an idempotent $u \in \mathbb{C}[X]$, let $S_u: \mathbb{C}[X] \to \mathbb{C}[X]$ be the $\mathbb{C}$-linear map $S_u(w)= wu$.

\begin{pro}
Let $X$ be a finite quandle and $u= \sum_{i=1}^n \alpha_i e_i$ an idempotent of $ \mathbb{C}[X]$. Then ${\rm Trace}(S_u)=  \sum_{i=1}^n \alpha_i |{\rm Fixed}(S_{e_i})|$, where ${\rm Fixed}(S_{e_i})$ is the set of basis elements fixed by $S_{e_i}$. Further, if $X$ is latin, then ${\rm Trace}(S_u)=\varepsilon(u)=0,1$.
\end{pro}

\begin{proof}
Note that each $S_{e_i}$ is a permutation of the basis elements, and $$S_u(w)=\sum_{i=1}^n \alpha_i (w e_i)= \sum_{i=1}^n \alpha_i S_{e_i}(w)= \big(\sum_{i=1}^n \alpha_i S_{e_i}\big)(w).$$
Thus, we have $${\rm Trace}(S_u)=\sum_{i=1}^n \alpha_i {\rm Trace}(S_{e_i})= \sum_{i=1}^n \alpha_i |{\rm Fixed}(S_{e_i})|.$$ In addition, if $X$ is latin, then $|{\rm Fixed}(S_{e_i})|=1$, and hence ${\rm Trace}(S_u)=\varepsilon(u)=0,1$.
\end{proof}

Next, we compute idempotents and Peirce spectra of complex quandle algebras of rank three. 

\begin{pro}\label{complex non-latin order 3 idempotents}
Let $X$ be the quandle with multiplication table 
$$X=
\begin{pmatrix} 
1 & 1 & 2 \\
2 & 2 & 1 \\
3 & 3 & 3 \\
\end{pmatrix}.
$$
Then $\mathcal{I}(\mathbb{C}[X]) = \big\{ (1-\beta) e_1 +  \beta e_2, ~ \alpha e_1 + \alpha e_2 +  (1 - 2 \alpha) e_3 ~\mid~ \beta, \alpha \in \mathbb{C} \big\}$.
\end{pro}

\begin{proof}
Let $u = \alpha e_1 + \beta  e_2 + \gamma e_3 \in \mathbb{C}[X]$ be a non-zero idempotent. Then the equality $u=u^2$ gives $$\alpha = \alpha^2 + \alpha \beta + \beta \gamma, \quad  \beta = \beta^2 + \alpha \beta +  \alpha \gamma \quad \textrm{and} \quad  \gamma = \gamma^2 + \alpha \gamma + \beta \gamma.$$
If $\gamma = \alpha=0$, then $\beta = 1$, and hence $u = e_2$. If $\gamma=0$ and $\alpha \neq 0$, then $\alpha = 1 - \beta$ and $u= (1 - \beta) e_1 +  \beta e_2$. Finally, if $\gamma \neq 0$, then the last equation of the above system gives $\gamma = 1 - \alpha - \beta$. Substituting value of $\gamma$ in the first equation gives $(\alpha - \beta) (\alpha + \beta) = (\alpha - \beta)$, which further implies that $\alpha = \beta$. Hence, we have $u= \alpha e_1 + \alpha e_2 +  (1 - 2 \alpha) e_3$, which completes the proof.
\end{proof}

\begin{pro}\label{complex dihedral idempotents}
Let $X$ be the quandle with multiplication table 
$$X=
\begin{pmatrix} 
1 & 3 & 2 \\
3 & 2 & 1 \\
2 & 1 & 3 \\
\end{pmatrix}.
$$
Then 
\begin{small}
$$\mathcal{I}(\mathbb{C}[X]) = \Big\{ e_1,~ e_2,~ e_3, ~\frac{1}{3}e_1+ \frac{1}{3}e_2+ \frac{1}{3}e_3,~ -\frac{1}{3}e_1 -\frac{1}{3}e_2+ \frac{2}{3}e_3,~ -\frac{1}{3}e_1 +\frac{2}{3}e_2- \frac{1}{3}e_3,~ \frac{2}{3}e_1 -\frac{1}{3}e_2- \frac{1}{3}e_3 \Big\}.$$
\end{small}
\end{pro}

\begin{proof}
Notice that $X \cong R_3$, the dihedral quandle of order three. Let $u = \alpha e_1 + \beta e_2 + \gamma e_3 \in \mathbb{C}[X]$ be a non-zero idempotent. Comparing coefficients of the basis elements in $u=u^2$ gives 
\begin{eqnarray}
\label{complex d1} \alpha &=& \alpha^2 + 2 \beta \gamma,\\
\label{complex d2} \beta &=& \beta^2 + 2 \gamma \alpha,\\
\label{complex d3} \gamma &=& \gamma^2 + 2 \alpha  \beta.
 \end{eqnarray} 
If $\varepsilon(u)= 0$, then $\gamma= -\alpha -\beta$. Substituting value of $\gamma$ in \eqref{complex d1} and \eqref{complex d2} gives  
$$\alpha = \alpha^2 - 2 \beta^2 -2 \alpha \beta \quad \textrm{and} \quad \beta = \beta^2 - 2 \alpha^2 -2 \alpha \beta.$$
Subtracting gives $\alpha-\beta = 3 (\alpha-\beta)(\alpha+\beta)$. If $\alpha=\beta$, then $\alpha= -3 \alpha^2$. Hence, we have $\alpha=\beta-\frac{1}{3}$ and $\gamma=\frac{2}{3}$. If $\alpha \neq\beta$, then  $\alpha+ \beta=\frac{1}{3}$. Substituting value of $\gamma$ in \eqref{complex d3} gives $\alpha \beta =-\frac{2}{9}$. Solving the equations gives $\alpha=\frac{2}{3}$ and $\beta=\gamma=-\frac{1}{3}$ or  $\alpha=\gamma=-\frac{1}{3}$ and $\beta=\frac{2}{3}$.
\par

Now, if $\varepsilon(u)= 1$, then $\gamma= 1-\alpha -\beta$. Substituting value of $\gamma$ in \eqref{complex d1} and \eqref{complex d2} gives  
$$\alpha = \alpha^2 - 2 \beta^2 -2 \alpha \beta + 2 \beta \quad \textrm{and} \quad \beta = \beta^2 - 2 \alpha^2 -2 \alpha \beta + 2 \alpha.$$
Subtracting gives $\alpha-\beta = (\alpha-\beta)(\alpha+\beta)$.  If $\alpha=\beta$, then $\alpha= 3 \alpha^2$. This gives $\alpha=\beta=0$ and $\gamma=1$ or $\alpha=\beta=\gamma=\frac{1}{3}$. If $\alpha \neq\beta$, then
$\alpha +\beta=1$, and hence $\gamma=0$. Using \eqref{complex d3}, we obtain either $\alpha=0$ and $\beta=1$ or  $\alpha=1$ and $\beta=0$. This completes the proof.
\end{proof}

By homogenising the idempotency equation $u=u^2$, the classical Bez\'out's theorem implies  that either there are infinitely many solutions or the number of distinct solutions is at most $2^{\dim_{\mathbb{C}}(A)}$. By including the zero idempotent in the counting, propositions \ref{complex non-latin order 3 idempotents} and \ref{complex dihedral idempotents} illustrate this dichotomy.

As a consequence of propositions \ref{complex non-latin order 3 idempotents} and \ref{complex dihedral idempotents}, we have

\begin{thm}\label{Peirce spectra 3 order quandle}
Let $X$ be a non-trivial quandle of order three.
\begin{enumerate}
\item If $X$ is non-latin, then $\sigma(\mathbb{C}[X])= \mathbb{C}$.
\item If $X$ is latin, then $\sigma(\mathbb{C}[X])= \{0, 1, -1 \}$.
\end{enumerate}
\end{thm}

\begin{proof}
Consider the ordering $e_1< e_2 < e_3$ of the basis of the quandle algebra. If $u=(1-\beta) e_1 +  \beta e_2$, then $S_u$ is the identity operator. If $u=\alpha e_1 + \alpha e_2 +  (1 - 2 \alpha) e_3$, then 
$S_u$ has eigenvalues 1, 1 and $4\alpha-1$. Hence, we obtain $\sigma(\mathbb{C}[X])= \mathbb{C}$.
\par
Since each $S_{e_i}$ is a permutation matrix, a routine check shows that $\sigma (e_1)=\sigma (e_2)= \sigma (e_3)=\{ 1, -1\}$. Further, direct computations give $\sigma \big(\frac{1}{3}e_1+ \frac{1}{3}e_2+ \frac{1}{3}e_3 \big)=\{0, 1\}$ and
$$\sigma \Big(-\frac{1}{3}e_1- \frac{1}{3}e_2+ \frac{2}{3}e_3 \Big)=\sigma \Big(-\frac{1}{3}e_1+ \frac{2}{3}e_2- \frac{1}{3}e_3 \Big)=\sigma \Big(\frac{2}{3}e_1 -\frac{1}{3}e_2- \frac{1}{3}e_3 \Big)=\{0, 1, -1\}.$$
Hence, we obtain $\sigma(\mathbb{C}[X])= \{0, 1, -1 \}$, which is desired.
\end{proof}

\begin{rmk}
Theorem \ref{Peirce spectra 3 order quandle}(1) shows that the Peirce spectrum of the complex quandle algebra can be as large as possible when the underlying quandle is non-latin. It is intriguing to know how arbitrary can the Peirce spectrum of the complex quandle algebra of a latin quandle be. Meanwhile, computer assisted computations show that  the set of idempotents of the complex quandle algebra of $R_5$, $R_7$ and $R_9$ consists of similar type of idempotents as for $R_3$, and that their Peirce spectra is again $\{0, 1, -1 \}$. 
\end{rmk}

In general, we believe that the following holds for each dihedral quandle of odd order.

\begin{conj}\label{Peirce spectra conjecture}
The Peirce spectrum of the complex quandle algebra of a dihedral quandle of odd order is $\{0, 1, -1 \}$.
\end{conj}

Any progress supporting Conjecture \ref{Peirce spectra conjecture} will add to the evidence towards Conjecture 3.10 of \cite{ENSS2022}.

\medskip

\textbf{Acknowledgement.} M.E. was partially supported by the Simons Foundation Collaboration Grant 712462. The work was carried out when M.S. was visiting the University of South Florida. He thanks the USIEF for the Fulbright-Nehru Academic and Professional Excellence Fellowship that funded the visit and the University of South Florida for the warm hospitality during his stay. M.S. also acknowledges support from the SwarnaJayanti Fellowship grant.

%%%%%%%%%%%%%%%%%%%%%%%%%%%%%%%%%%%%%%%%


\begin{thebibliography}{99}

\bibitem{MR3977818} V. G. Bardakov, I. B. S. Passi and M. Singh, \textit{Quandle rings}, J. Algebra Appl. 18 (2019), 1950157, 23 pp.
		
\bibitem{BPS1} V. G. Bardakov, I. B. S. Passi and M. Singh, \textit{Zero-divisors and idempotents in quandle rings},  Osaka J. Math. 59 (2022), 611--637.

\bibitem{MR3197054} A. S. Crans and S. Nelson, \textit{Hom quandles}, J. Knot Theory Ramifications 23 (2014), no. 2, 1450010, 18 pp.

\bibitem{MR1954330} M. Eisermann, \textit{Homological characterization of the unknot}, J. Pure Appl. Algebra 177 (2003), no. 2, 131--157. 
 
\bibitem{MR3205568} M. Eisermann, \textit{Quandle coverings and their Galois correspondence}, Fund. Math. 225 (2014), no. 1, 103--168.	
	
\bibitem{MR3915329} M. Elhamdadi, N. Fernando and B. Tsvelikhovskiy, \textit{Ring theoretic aspects of quandles}, J. Algebra. 526 (2019), 166--187.
	
	
\bibitem{EMSZ2022}	M. Elhamdadi, A. Makhlouf, S. Silvestrov and E. Zappala, \textit{The derivation problem for quandle algebras}, Internat. J. Algebra Comput. 32 (2022), 985--1007.


\bibitem{ENSS2022} M. Elhamdadi, B. Nunez, M. Singh and D. Swain, \textit{Idempotents, free products and quandle coverings}, (2022), arXiv:2204.11288v2.

\bibitem{MR1194995} R. Fenn and C. Rourke, \textit{Racks and links in codimension two}, J. Knot Theory Ramifications 1 (1992), no. 4, 343--406.

\bibitem{MR2628474} D. E. Joyce, \textit{An algebraic approach to symmetry with applications to knot theory}, Ph.D. Thesis, University of Pennsylvania, 1979.

\bibitem{MR4191467} Y. Krasnov and V. G. Tkachev, \textit{Variety of idempotents in non-associative algebras}, Topics in Clifford analysis--special volume in honor of Wolfgang Spr\"ossig, 405--436, Trends Math., Birkh\"auser/Springer, 2019.


\bibitem{Maple} Maple 15 - Magma package-copywrite by Maplesoft, a division of Waterloo Maple, Inc, 1981--2011.

\bibitem{MR0672410} S. V. Matveev, \textit{Distributive groupoids in knot theory}, Mat. Sb. (N.S.) 119(161) (1982), no. 1, 78--88, 160.

\bibitem{MR3243534} N. Nadirashvili, V. Tkachev and S. Vl\u adu\c t, \textit{Nonlinear elliptic equations and non-associative algebras}, Mathematical Surveys and Monographs, 200. American Mathematical Society, Providence, RI, 2014. viii+240 pp. 


\bibitem{MR0210757} R. D.  Schafer, \textit{An introduction to non-associative algebras}, Pure and Applied Mathematics, Vol. 22 Academic Press, New York-London 1966 x+166 pp. 


\bibitem{knot info} https://knotinfo.math.indiana.edu.

\end{thebibliography}
\end{document}